\documentclass{amsart}
\usepackage{amssymb}
\usepackage{amsmath}
\usepackage{color}
\def\indiq{{\bf 1}}\def\E{{\mathbb E}}
\def\P{{\mathbb P}}
\def\R{{\mathbb R}}
\def\D{{\mathbb D}}

\def\N{{\mathbb N}}
\def\intot{\int_0^t}
\def\cP{{\mathcal P}}
\def\cL{{\mathcal L}}
\def\cF{{\mathcal F}}
\def\cS{{\mathcal S}}
\def\cB{{\mathcal B}}
\def\cW{{\mathcal W}}

\def\vip{\vskip0.2cm}
\def\sg{{\rm sg}}
\def\e{\varepsilon}
\def\bN{{\mathbf N}}

\newtheorem{theo}{Theorem}
\newtheorem{prop}[theo]{Proposition}

\newtheorem{rem}[theo]{Remark}
\newtheorem{lem}[theo]{Lemma}

\newtheorem{ass}[theo]{Assumption}

\begin{document}

\title{On a toy model of interacting neurons}
\author{Nicolas Fournier}
\author{Eva L\"ocherbach}

\address{N. Fournier: Laboratoire de Probabilit\'es et Mod\`eles al\'eatoires, 
UMR 7599, UPMC, Case 188, 4 place Jussieu, F-75252 Paris Cedex 5, France.}

\email{nicolas.fournier@upmc.fr}

\address{E. L\"ocherbach: CNRS UMR 8088, D\'epartement de Math\'ematiques, Universit\'e de Cergy-Pontoise,
2 avenue Adolphe Chauvin, 95302 Cergy-Pontoise Cedex, France.}

\email{eva.loecherbach@u-cergy.fr}



\begin{abstract}
We continue the study of a stochastic system of interacting neurons introduced in
De Masi, Galves, L\"ocherbach and Presutti \cite{aaee}. 
The system consists of $N$ neurons, each spiking randomly with rate depending on its membrane potential.
At its spiking time, the neuron potential is reset to $0$ and all other neurons
receive an additional amount $1/N$ of potential. Moreover, electrical synapses induce a deterministic drift of 
the system towards its average potential.
We prove propagation of chaos of the system, as $N \to \infty$, to a limit nonlinear
jumping stochastic differential 
equation. We consequently improve on the results of \cite{aaee}, since (i) we remove the
compact support condition on the initial datum, (ii) we get a rate of convergence in $1/\sqrt N$.
Finally, we study the limit equation: we describe the shape of its time-marginals,
we prove the existence of a unique non-trivial invariant distribution,
we show that the trivial invariant distribution is not attractive, and in a special case, 
we establish the convergence to equilibrium.
\end{abstract}

\maketitle 

\section{Introduction and main results}

\subsection{The model} We consider, for each $N\geq 1$, a family of i.i.d. 
Poisson measures
$(\bN^i(ds, dz))_{i=1,\dots,N}$ on $\R_+ \times \R_+ $ having intensity measure $ds dz$,
as well as a family $(X^{N,i}_0)_{i=1,\dots,N}$ of $\R_+$-valued random variables 
independent of
the Poisson measures.
The object of this paper is to study the Markov process
$X^N_t = (X^{N, 1 }_t, \ldots , X^{N, N}_t )$
taking values in $\R_+^N$ and solving, for $i=1,\dots,N$, for $t\geq 0$,
\begin{align}\label{eq:dyn}
X^{N, i}_t = & X^{N,i}_0  - \lambda \intot ( X^{N, i}_s - \bar X^N_s) ds -  \intot \int_0^\infty 
X^{N, i}_{s-}  \indiq_{ \{ z \le  f ( X^{N, i}_{s-}) \}} \bN^i (ds, dz) \\
&+ \frac1N \sum_{ j \neq i } \intot\int_0^\infty \indiq_{ \{ z \le  f ( X^{N, j}_{s-}) \}} \bN^j (ds, dz),
\nonumber
\end{align}  
where $ \bar X^N_t = N^{-1} \sum_{j=1}^N X^{N, j}_{t}$. The coefficients of this system are
$\lambda \geq 0$ and a function $f:\R_+\mapsto \R_+$  
satisfying (at least) the following assumption.

\begin{ass}\label{ass:1}
$f$ is non-decreasing, $f(0) = 0$, $f(x)>0$ for all $x>0$, $\lim_\infty f = \infty$ 
and $f\in C^1(\mathbb R_+)$.
\end{ass}


\begin{prop}\label{theo:1}
Grant Assumption \ref{ass:1}, let $\lambda\geq 0$ and $N\geq 1$ be fixed and suppose that 
$\bar X_0^{N } < \infty$ a.s. Then there exists a unique c\`adl\`ag adapted 
strong solution  $(X^N_t)_{t\geq 0}$ 
to \eqref{eq:dyn} taking values in $\R_+^N$.
\end{prop}

\subsection{Formal description and goals}
This paper continues a study started in De Masi, Galves, L\"ocherbach and Presutti \cite{aaee}. 
The particle system \eqref{eq:dyn} is the model of interacting neurons considered in \cite{aaee}, 
inspired by a work of Galves and L\"ocherbach \cite{ae}. 
The system is made of $N$ neurons. Each $X_t^{N, i }$ models the 
membrane potential at time $t$ of the $i$-th neuron. Interactions between neurons are due to two types
of synapses, chemical and electrical synapses. Chemical synapses are characterized through {\it spiking}
of the neurons, i.e.\ a fast trans-membrane current. Spiking occurs randomly following a point process of rate $f(x)$ 
for a neuron of which the membrane potential equals $x$. At its spiking time, the membrane potential 
of the neuron is reset to a resting potential which we choose to be equal to $0.$ At the same
time, the action of the chemical synapses induces an increase of the membrane potentials of 
the other neurons: they receive an additional amount $1/N$ of potential. Our model does not take into account a 
refractory period. Electrical synapses, which are due to gap junctions, work constantly over time and tend to 
synchronize the membrane potentials of the neurons. Such electrical synapses are typical for systems requiring 
fast responses to stimuli, often found in animals. They induce a constant drift of the system towards the 
average membrane potential of the system, at speed $\lambda$.

\vip

It is well known that neuronal interactions can exhibit very complicated interaction 
structures. Our model only takes into account the
average effect of the interactions. We are thus working 
with a toy model where interactions are of
{\it mean-field} type.

\vip

Concerning $f$, we think of functions of the type $f(x)=(x/x_0)^\xi$ with $\xi$ large and some {\it soft threshold} $x_0$. In this case, for $x$ the membrane potential of the neuron, 
spiking occurs at very 
low rate if $x<x_0$ and with very high rate if $x>x_0$.
Note that in the biological literature it is often assumed that spiking occurs when the membrane potential 
reaches a fixed threshold $x_0$, which would correspond to $f(x) = \infty 1_{[x_0, \infty)}$.
However, a well-defined fixed threshold does not seem to exist in {\it in vivo} neurons, 
see e.g.\ Jahn, Berg, Hounsgaard and Ditlevsen \cite{susanne} who propose a statistical study showing that a point 
process model in which the jump intensity depends on the membrane potential is well-adapted.
We therefore propose a smooth firing rate depending on the membrane potential of the form 
$f(x)=c x^\xi$ with $\xi$ quite large and $c>0$.

\vip

We are interested in the evolution of a large system of neurons, i.e.\ in the 
limit $N\to\infty$. We prove a weak law of large numbers for the empirical 
measure of the system ({\it propagation of chaos}): we show that
the empirical distribution of the system becomes
deterministic as $N\to \infty$ and tends to the law of a limit process $(Y_t)_{t \geq 0} $ which 
solves a {\it nonlinear} jumping SDE.

\vip

Such a result has already been achieved in \cite{aaee} in the case of a {\it compact support}, 
i.e. when the initial conditions $X_0^{N, i }$ are uniformly bounded. It is then 
possible to control the evolution of the support of the law of the 
process over time. Consequently, the propagation of chaos can be shown for any locally Lipschitz continuous
function $f$, exactly as if it was globally Lipschitz continuous and bounded.

\vip

The case where the initial conditions are not compactly supported is more delicate, at least
when $f$ is not globally Lipschitz continuous. 
Our results work under quite weak moment conditions on the
initial datum, for quite general functions $f$. 
We obtain a rate of convergence in $1/\sqrt N$, as 
one expects. These results should remain true when adding a diffusive component 
to the dynamics of individual neurons, at the cost of a higher level of technicality.

\vip

Finally, we propose a short study of the limit equation.
We describe the shape of its time-marginals,
we prove the existence of a unique non-trivial invariant distribution, we show that the trivial
invariant distribution is not attractive, and when $\lambda=0$, 
we establish the convergence to equilibrium for a class of initial conditions.

\vip

Let us mention that all the results and proofs below have been elaborated thinking of the case where
$f(x)=x^\xi$ with $\xi\geq 2$, which thus satisfies all the conditions of the paper.

\subsection{References}

Using a mean-field approach in order to describe the typical behavior of a neuron within a large population 
of similarly behaving neurons from a macroscopic point of view is by now classical in neuromathematics. 
A lot of effort has been spent by the neuromathematical community focussing on the study of leaky 
integrate-and-fire models and their mean-field limits; in these models the membrane potential of a 
neuron is described by a (finite or infinite dimensional) diffusion process, and spiking occurs when 
reaching a threshold. 

\vip

Recent interesting papers using a fixed threshold are those of Delarue, Inglis, Rubenthaler 
and Tanr\'e \cite{d2015a} and 
\cite{d2015b}, see also C\'aceres, Carrillo and Perthame 
\cite{ccp}. Here, the membrane potential is described by a one dimensional diffusion process. The existence of a 
fixed threshold may lead to severe mathematical problems related to a possible blow-up of the limit nonlinear 
equation. Such a blow-up appears when a  macroscopically large proportion of neurons spike at the same time. 
Avalanches and synchronization are phenomena which are related to such a blow-up. 

\vip

Notice that on the contrary, our model does not include a diffusive part in the evolution of each neuron's 
membrane potential. As indicated above, we expect our result on propagation of chaos to remain true when adding 
such a diffusive component. But of course none of the above mentioned phenomena such as blow-up or avalanches appear 
in our model since spiking occurs at a smooth rate which is finite all over the state space. As we have already 
mentioned, this choice of modeling is motivated by biological considerations. Consequently, we do not have to face 
the same difficulties. The problems we have to deal with are linked to the jump part of the equation, more 
specifically to the fact that the spiking rate is not globally Lipschitz. 

\vip

Recently, Inglis and Talay \cite{it} have proposed a model of the integrate and fire type where neurons do spike 
when hitting a fixed threshold but where the effect of a spike is not instantaneously transmitted to the other 
neurons. As a consequence, their model does not present blow-up phenomena neither. 

\vip

For an excellent overview of the mean-field approach in integrate-and-fire models with a strong modeling point 
of view, we refer the reader to Faugeras, Touboul and Cessac \cite{faugeras} . They specifically deal with the 
case where several big populations of neurons interact through their neural efficacities which are chosen to be 
i.i.d. Gaussian random variables. In particular, they also deal with negative 
synaptic weights. Actually the extension of our model to the multi-population case, including also inhibitory 
synapses, seems to be quite straightforward and is part of a future work. 

\vip

Finally, in a recent article, Lu\c{c}on and Stannat \cite{ls} consider a population of mean field interacting 
diffusions which are attached to spatial positions and evolve within a random environment. Here, the spiking is 
encoded within the diffusion model (as in the Fitzhugh-Nagumo model). 
The interaction strength between two 
neurons depends on their spatial distance and may show singularities. Moreover, the coefficients of the 
underlying diffusion are of polynomial growth and therefore not globally Lipschitz neither. However, the 
techniques and results obtained in this article are clearly far from the considerations we are interested in in the present paper. 

\vip
To summarize, our aim is not to build a model 
which describes the full neurophysiological reality, but to study a simple model 
describing some basic biological features and to concentrate on the randomness hidden behind the 
spike times. It is inspired by integrate-and-fire models, but
spiking occurs randomly, with state-dependent intensity and the system does not contain any other 
source of randomness. Let us finally mention that this model can also be interpreted in terms of an associated 
nonlinear Hawkes process including a variable memory structure.
We refer to Hansen, Reynaud-Bouret and Rivoirard \cite{rbr} for an interesting statistical study 
of the neuronal interaction graph using Hawkes processes. 

\vip
 
From the purely probabilistic point of view, propagation of chaos is a popular topic
since the seminal works of Kac \cite{k}, McKean \cite{mk1,mk2}
and Sznitman \cite{s2,s}.
Generally, one tries to prove that the time-evolution of a particle, 
interacting with a large number of other particles,
can be approximated by a {\it nonlinear} process. By {\it nonlinear}, in the sense of McKean,
we mean that the law of the process itself is involved in its dynamics.
There exist essentially two kinds of proofs. 

\vip

$\bullet$ The first one, based on {\it coupling} and
often used in \cite{s}, provides a (often optimal) rate of convergence but works mainly 
when all the parameters of the model are globally Lipschitz continuous.
However, it sometimes happens that the non-Lipschitz terms are not really a problem, when they have, roughly,
the good sign: see Malrieu \cite{m}, who studies some McKean-Vlasov equation with, roughly, a convex interaction
potential.
More recently, it has been shown in Bolley, Ca\~nizo and Carrillo \cite{bcc} that the coupling method can also
apply to the case of non-globally Lipschitz parameters, under some very restrictive exponential
moment conditions. They also get a (almost optimal) rate of convergence. This idea
has also been exploited for the Boltzmann equation in \cite{fm}.

\vip

$\bullet$ The second method, elaborated in \cite{s2} when studying the Boltzmann equation, 
is based on tightness/consistency/uniqueness of the nonlinear process. It applies
much more generally (it requires only some {\it a priori} bounds and some continuity of the parameters),
but does not provide any rate of convergence. 

\vip

In the present paper we make use of the two methods and investigate to which extent they can be applied. Roughly, the tightness/consistency/uniqueness 
works very well, under some very light assumptions on $f$ and on the initial conditions.
But the most important point of the paper is that, still for quite a general class of functions $f$ 
(as $x^\xi$ with $\xi\geq 2$), we show that the coupling method
also works, without imposing some exponential moment conditions. 
This is very specific to the model under study,
relies on quite fine computations, and on the use of an {\it ad hoc} distance. As previously mentioned
we get an optimal rate of convergence. 

\vip

Finding an {\it ad hoc} distance is a classical strategy to prove uniqueness of the solution or to study
its large time behavior, in all fields of differential equations.
It is a good approach, in the sense that it often allows for many developments, such as stability
and convergence of approximate models. But each model requires its own study and the {\it good} distance
often looks mysterious. The {\it distance} may or may not depend on the precise parameters of the model.
Let us quote a few papers.
For example, Tanaka \cite{t,t2} discovered, using a specific nonlinear jumping SDE, that the Wasserstein distance 
with quadratic cost between two solutions of the homogeneous Boltzmann equation for Maxwell molecules is decreasing,
providing the first uniqueness result for the Boltzmann equation in the 
physically reasonable case without cutoff. 
Bolley, Guillin and Malrieu \cite{bgm} were able to precisely study, using a nonlinear Brownian SDE, 
the large-time behavior of solutions
to a Vlasov-Fokker-Planck equation by introducing an {\it ad hoc} modification of the Wasserstein distance
depending on the parameters of the equation.
They also quantified, with similar tools, the convergence of some particle systems.
In \cite{f,fl}, we introduced an {\it ad hoc} distance to prove uniqueness of some infinite stochastic interacting 
particle systems undergoing coalescence. Here also, the distance was depending on the interaction
kernel.

\vip

However, the study proposed in the present paper, and in particular the proof of the uniqueness  of the limit 
equation, is situated in a completely different mathematical framework compared to the above mentioned papers, 
and the specific choice of an {\it ad hoc} distance that we propose is a new feature.

\subsection{The limit equation}
Assume that the $X^{N,i}_0$ are i.i.d. with common law $g_0$ on $\R_+$.
Simple considerations show that 
the solution $(X^N_t)_{t\geq 0}$ should behave, for $N$ large, as $N$ independent copies
of the solution to the following {\it nonlinear}, in the sense of McKean, SDE. Let $Y_0$
be a $g_0$-distributed random variable, 
independent of a Poisson measure $\bN(ds,dz)$ on 
$\R_+ \times \R_+ $ having intensity measure $ ds dz$. An $\R_+$-valued c\`adl\`ag
adapted process $(Y_t)_{t\geq 0}$ is said to solve the nonlinear SDE if 
\begin{equation}\label{eq:dynlimit}
Y_t = Y_0 - \lambda \intot (Y_s - \E[ Y_s])ds 
- \intot\int_0^\infty Y_{s- } \indiq_{ \{ z \le  f ( Y_{s-}) \}} \bN  (ds, dz) + \intot \E[f(Y_s)]ds.
\end{equation} 

For PDE specialists, let us mention that for $(Y_t)_{t\geq 0}$ a solution to \eqref{eq:dynlimit},
$g(t)=\cL(Y_t)$ solves the following nonlinear PDE in weak form: for any 
$ \phi \in C^1_b (\R_+),$ the set of $C^1$-functions on $[0,\infty)$ such that $\phi$ and $\phi'$ are bounded,
for any $t\geq 0$,
\begin{align*}
\int_0^\infty \phi (x) g (t, dx) =& \int_0^\infty \phi (x) g (0, dx) +\intot \int_0^\infty 
\Big([ \phi ( 0) - \phi ( x) ] f(x) + \phi'(x) [ a_s - \lambda x ]\Big)g (s, dx)ds,
\end{align*}
where $a_t=\int_0^\infty [f(x)+\lambda x]g(t,dx)$. 
Setting also $p_t=\int_0^\infty f(x)g(t,x)dx$,
the strong equation writes
\begin{align*}
\forall\; t>0,\; \forall \;x>0, \;\; 
\partial_t g(t,x) = [\lambda x - a_t]\partial_x g(t,x) + [\lambda - f(x)]g(t,x) \quad 
\hbox{and}\quad \forall \;t>0, \;\;
g(t,0)=p_t/a_t,
\end{align*}
with $g(0,x)$ a given probability density on $[0,\infty)$.

\vip

The nonlinear SDE \eqref{eq:dynlimit} is not clearly well-posed, unless one assumes e.g.\ that 
$f$ is globally Lipschitz-continuous and bounded. Under Assumption \ref{ass:1},
we are generally only able to check the weak existence, that is
existence of a filtered probability space on which there is a Poisson measure $\bN$ and 
a c\`adl\`ag adapted
process $(Y_t)_{t\geq 0}$ such that \eqref{eq:dynlimit} holds true for all $t\geq 0$.

\begin{ass}\label{ass:2}
$f \in C^2(\R_+)$ is convex increasing and 
$\sup_{x\geq 1} [f'(x)/f(x)+f''(x)/f'(x)]<\infty$.
\end{ass}

\begin{theo}\label{theo:2}
Grant Assumption \ref{ass:1} and suppose that $\lambda \geq 0$.

\vip

(i) Assume only that $\E[Y_0]<\infty$. Then there is weak existence of a solution
$(Y_t)_{t\geq 0}$ to \eqref{eq:dynlimit} satisfying 
$\int_0^t \E[Y_sf(Y_s)]ds <\infty$ for all $t\geq 0$.

\vip

(ii) Assume now that the law of $Y_0$ is compactly supported. Then there exists a path-wise unique
solution $(Y_t)_{t\geq 0}$ to \eqref{eq:dynlimit} such that there is a deterministic locally bounded function 
$A:\R_+\mapsto\R_+$ such that a.s., $\sup_{[0,\infty)} (Y_t/A(t))<\infty$.

\vip

(iii) Grant now Assumption \ref{ass:2} and assume that $\E[f(Y_0)]<\infty$. Then there is a
path-wise unique solution to \eqref{eq:dynlimit} satisfying 
$\sup_{[0,t]} \E[f (Y_s) ] < \infty$ for all $t\geq 0$.
\end{theo}

Let us mention that Assumption \ref{ass:2} can be slightly relaxed: if for example $f=f_1+f_2$ 
with $f_1$ satisfying Assumptions \ref{ass:1} and \ref{ass:2} and $f_2(x)=\int_0^x \psi(y)dy$ 
with $\psi \in C^1_c([0,\infty))$ nonnegative, then Theorem \ref{theo:2}-(iii) still holds true. In fact, what 
we really need
is that the conclusions of Lemma \ref{lem:1} below are satisfied.

\vip

Let us comment on the results of Theorem \ref{theo:2}. Point (i) is not hard: it is checked by compactness and is actually a consequence
of Theorem \ref{theo:6}-(i)-(ii) below. The only noticeable point
is that the condition $\E[Y_0]<\infty$ is sufficient to guarantee that indeed, 
$\int_0^t \E[Y_sf(Y_s)]ds <\infty$ for all $t\geq 0$, which 
is sufficient to handle a proof by compactness. 
Point (ii) is not very complicated and has already been proven in \cite{aaee}.
The only difficult point is to check that if $Y_0$ is bounded, then
$Y_t$ is {\it a priori} bounded for all $t$. Once this is seen, the function $f$ can be considered as if it was 
bounded and globally Lipschitz continuous. Finally, (iii) is much more delicate and goes clearly beyond the 
results of \cite{aaee}. Indeed, when computing the time derivative of $\E[|X_t-Y_t|]$, for $X$ and $Y$
two solutions to \eqref{eq:dynlimit}, some nonlinear terms appear:
there is no hope to conclude uniqueness by the Gronwall lemma.
One possibility is to use the famous $x|\log x|$ extension of the Gronwall lemma, but this requires
to have some bounds for something like $\sup_{[0,T]}\E[\exp(f(Y_t))]<\infty$, 
see Bolley, Ca\~nizo and Carrillo \cite{bcc} or \cite{fm} for such considerations, but this is
not very satisfying, since it requires the strong condition that $\E[\exp(f(Y_0))]<\infty$.
We thus search for a more convenient ``distance''. We first observe that when time-differentiating 
$\E[|f(X_t)-f(Y_t)|]$, the contribution of the most unpleasant term (the Poisson integral) is non-positive:
it gives exactly $-\E[|f(X_t)-f(Y_t)|^2]$, which is a very good point. However, the other terms cause problems
for small values of $X$, $Y$, if $f$ vanishes too fast at $0$ (e.g. $f(x)=x^\xi$ with $\xi \geq 2$). 
To overcome this difficulty, it actually suffices to work with $\E[|H(X_t)-H(Y_t)|]$, 
where $H(x)\simeq f(x)+ x\land 1$. Of course, it is more convenient to
use a smooth version of $x\land 1$, so that we will work with 
$H(x)=\arctan x + f(x)$. It is likely that we could also use $H(x)=\ell(x) + f(x)$,
with any smooth increasing function $\ell(x)$ behaving like $a x$ near $0$ (for some $a>0$) 
and tending to some constant $b>0$ as $x\to\infty$.

\subsection{Propagation of chaos}

We start with a general weak result. The set $\D(\R_+)$ of c\`adl\`ag functions on $\R_+$
is endowed with the topology of the Skorokhod convergence on compact time intervals,
see Jacod and Shiryaev \cite{js}.

\begin{theo}\label{theo:6}
Grant Assumption \ref{ass:1} and suppose that $\lambda \geq 0$.
Consider a probability distribution $g_0$
on $\R_+$ such that $\int_0^\infty y g_0(dy)<\infty$. For each $N\geq 1$,
consider the unique solution $(X^N_t)_{t\geq 0}$ to \eqref{eq:dyn} starting from
some i.i.d. $g_0$-distributed initial conditions $X^{N,i}_0$.

\vip

(i) The sequence of processes $(X^{N,1}_t)_{t\geq 0}$ is tight in $\D(\R_+)$.

\vip

(ii) The sequence of empirical measures $\mu_N=N^{-1}\sum_{i=1}^N \delta_{(X^{N,i}_t)_{t\geq 0}}$
is tight in $\cP(\D(\R_+))$.

\vip

(iii) Any limit point $\mu$ of $\mu_N$ a.s. belongs to $\cS:=\{\cL((Y_t)_{t\geq 0}) \, : \, (Y_t)_{t\geq 0}$
solution to \eqref{eq:dynlimit} with $\cL(Y_0)=g_0$ and satisfying $\int_0^t \E[Y_sf(Y_s)]ds <\infty$ 
for all $t\geq 0\}$.

\vip

(iv) If moreover (a) $g_0$ is compactly supported or (b)
$\int_0^\infty f(y)g_0(dy)<\infty$ and $f$ satisfies Assumption \ref{ass:2}, 
then $\mu_N$ goes in probability to 
$\cL((Y_t)_{t\geq 0})$, where $(Y_t)_{t\geq 0}$ is the unique solution to \eqref{eq:dynlimit}.
\end{theo}

Points (i), (ii) and (iii) are not very difficult. The fact that $f$ is not globally Lipschitz continuous
is not really a problem when working by compactness. And of course, 
point (iv), which is usually called propagation of chaos, is a consequence of points (ii) and (iii)
and of the uniqueness results of Theorem \ref{theo:2}. Again, the above theorem has already been proven 
in the case
of compact support in \cite{aaee}; but the techniques employed in \cite{aaee} cannot be used in the 
general case where 
$g_0$ is not compactly supported. 
Under a few additional conditions, we get a 
quantified version of the above convergence, at least concerning the time marginals.

\begin{ass}\label{ass:3}
There is a constant $C$ such that for all $x,y\in\R_+$, $f(x+y)\leq C(1+f(x)+f(y))$.
\end{ass}

\begin{theo}\label{theo:4}
Grant Assumptions \ref{ass:1}, \ref{ass:2} and \ref{ass:3} and suppose that $\lambda \geq 0$ and that 
$\int_0^\infty f^2(y) g_0(dy)<\infty$. Consider, for each $N\geq 1$, 
the unique solution $(X^N_t)_{t\geq 0}$ to \eqref{eq:dyn} starting from
some i.i.d. $g_0$-distributed initial conditions $X^{N,i}_0$. Consider also
the unique solution $(Y^{N,1}_t)_{t\geq 0}$ to \eqref{eq:dynlimit}
starting from $Y_0=X^{N,1}_0$ and driven by the Poisson
random measure $\bN^1(ds,dz)$. The law of $(Y^{N,1}_t)_{t\geq 0}$ does not depend on $N$, and we denote by $g(t)
:=\cL(Y_t^{N,1})$.
Introduce  $H(x) = f(x) + \arctan (x)$. Then
for all $T>0$, there is a constant $C_T$, depending only on 
$T$, $\lambda$, $f$ and $\int_0^\infty f^2(y)g_0(dy)$ such that
$$
\sup_{[0,T]} \Big(\E\big[|X^{N,1}_t-Y^{N,1}_t|\big] + \E\big[|H(X^{N,1}_t)-H(Y^{N,1}_t)|\big]\Big) 
\leq \frac{C_T}{\sqrt N}.
$$
Assume furthermore that $\int_0^\infty y^{2+\e} g_0(dy)<\infty$ for some $\e>0$. Then for all $T>0$,
there is a constant $C_T$, depending only on 
$T$, $\lambda$, $f$, $\e$ and $\int_0^\infty [f^2(y)+y^{2+\e}]g_0(dy)$ such that
$$
\sup_{[0,T]} \E\Big[\cW_1\Big(\frac 1N \sum_{i=1}^N \delta_{X^{N,i}_t},g(t) \Big)\Big]\leq \frac{C_T}{\sqrt N}.
$$
\end{theo}

The Monge-Kantorovich-Wasserstein distance between two probability measures $\mu$ and 
$\nu$ on $\R_+$ with finite expectations is defined by $\cW_1(\mu,\nu)=\inf\{\E[|U-V|]$, $\cL(U)=\mu$ and 
$\cL(V)=\nu\}$.
The moment condition $\int_0^\infty f^2(y) g_0(dy)<\infty$ is very reasonable: somewhere in the proof,
we will have to study the convergence of $N^{-1}\sum_{j=1}^N f(Y^{N,j}_t)$ to $\E[f(Y_t^{N,1})]$.
If we want a rate of convergence of order $1/\sqrt N$, such an assumption is needed.

\subsection{Large time behavior of the limit process}

First, we study the possible invariant measures.

\begin{theo}\label{theo:41}
Grant Assumption \ref{ass:1} and let $\lambda\geq 0$. Then 
the nonlinear equation \eqref{eq:dynlimit} has exactly two invariant probability measures
supported in $\R_+$. The first one is $\delta_0 $. The second one is of the form $g(dx)=g(x)dx$,
with $g:[0,\infty) \mapsto [0,\infty)$ defined as follows.

\vip

(i) If $\lambda > 0 ,$ then
$$ g (x) = \frac{p}{p + \lambda m - \lambda x } 
\exp \Big( - \int_0^x \frac{f(y) }{p + \lambda ( m - y ) } dy \Big) 
\indiq_{\{ 0 \le x < m + p/ \lambda \} },$$
where $p>0$ and $m>0$ are uniquely determined by the constraints
$\int_0^\infty g(dx)= 1$, $\int_0^\infty x g(dx)= m$.
Furthermore, we have $\int_0^\infty f(x) g(dx) = p$ and  $m +p/\lambda>1$.

\vip

(ii) If $\lambda =0 ,$ then 
$$ g(x) = \exp \Big( - \frac{1}{p}\int_0^x f(y) dy \Big) ,$$
where $p >0 $ is uniquely determined by the constraint $\int_0^\infty g(x) dx = 1 $. Furthermore,
it holds that $\int_0^\infty f(x)g(x) dx = p.$
\end{theo}

Starting from a (reasonable) non-trivial
initial condition, it is likely that $Y_t$ goes in law to $g$ as $t\to \infty$.
When $\lambda=0$, we can prove such a result under a few assumptions.

\begin{prop}\label{formal}
Grant Assumptions \ref{ass:1} and \ref{ass:2} and assume that $\lambda=0$. 
Suppose moreover that the solution $(Y_t)_{t\geq 0}$ to \eqref{eq:dynlimit} 
starts from $Y_0 \sim g_0 (x) dx $ where $g_0 \in C^1_b([0,\infty))$ satisfies 
$g_0(0)=1$, $\int_0^\infty f^2 (x) g_0 (x) dx 
< \infty $ and $\int_0^\infty |g_0'(x)| dx < \infty$.  
Denote by $g(t)$ the law of $Y_t$ and write $g$ for the invariant probability measure
defined in Theorem \ref{theo:41}-(ii). Then we have $\lim_{t\to\infty} \|g(t )-g\|_{TV}=0,$ where 
$ \|  \cdot \|_{TV} $ denotes the total variation distance. If furthermore there are $c>0$ and  
$\xi \geq 1$ such that $f(x)\geq c x^\xi$ for all $x\in [0,1]$, then we have the estimate
$\|g (t) - g \|_{TV} \leq C (1+t)^{-1/\xi}$.
\end{prop}

Our proof, which is probably not optimal, 
relies on the use of the {\it strong version} of the PDE satisfied by $g(t)$. The 
regularity conditions on $g_0$, as well as the {\it structure} condition $g_0(0)=1$, will imply
that $g(t)$ has a sufficiently regular density. As can be seen in the next subsection
(see also \eqref{tacc} in Subsection \ref{xxx}),
if $g_0(0)\ne 1$, then 
$g(t,y)$ will be discontinuous for all $t>0$ 
(it will have one jump at some value $y_t \in (0,\infty)$ depending on $t$).

\vip

When $\lambda>0$, the situation is more intricate and we have not been able to prove the convergence
to equilibrium. One reason is that the non-degenerate invariant probability measure is more complicated,
compactly supported and possibly not continuous at the right extremity of its support. 
In any case, the computation handled to treat the case $\lambda=0$ does not extend.
A natural approach would be to show first that $\lim_{t\to\infty} \E[\lambda Y_t+f(Y_t)]$ exists.
However, $\E[\lambda Y_t+f(Y_t)]$ does not solve a closed equation, and we did not succeed.
The only result we are able to prove is that 
$Y_t$ cannot go in law to the invariant measure $\delta_0$. 
Our proof, which was as usual elaborated in the case where $f(x)=x^\xi$, actually extends to the following
situation.

\begin{ass}\label{ass:4}
(i) It holds that $\limsup_{x\to\infty} [f'(x)/f(x)] < 1$.

\vip

(ii) There are $\xi\geq 1$, $\zeta \geq \xi -1$ and some constants $0<c<C$ such that 
$cx^\xi \leq f(x) \leq C(x^{\xi-1}+x^\zeta)$.
\end{ass}

\begin{prop}\label{paszero}
Let $\lambda\geq 0$ and grant Assumptions \ref{ass:1} and \ref{ass:2}. Assume that
$\P(Y_0=0)<1$ and that $\E[f^2(Y_0)]<\infty$ and 
consider the unique solution $(Y_t)_{t\geq 0}$ to \eqref{eq:dynlimit}.
If $\lambda>0$, grant moreover Assumption \ref{ass:4} and suppose that $\E[Y_0^{\zeta+1}]<\infty$.
Then
$Y_t$ does not go to $0$ in law as $t\to\infty$.
\end{prop}

\subsection{Shape of the time-marginals of the nonlinear SDE} 
The next theorem shows that random spiking creates density near $0$, even if the system starts
from a singular initial condition (see also Theorem 2 of \cite{aaee} in the case of a smooth initial condition). 

\begin{theo}\label{theo:limitdensity}
Let $\lambda\geq 0$, grant Assumptions \ref{ass:1} and \ref{ass:2} and suppose that $\E[f^2(Y_0)]<\infty$ and 
that $\P(Y_0=0)<1$. Consider the unique solution $(Y_t)_{t\geq 0}$ to \eqref{eq:dynlimit},
set $p_t=\E [ f(Y_t) ]$, $a_t = \lambda \E [ Y_t] + \E [ f(Y_t) ]$ and denote by $g(t)$ the law of $Y_t$.
The functions $t\mapsto a_t$ and $t\mapsto p_t$ are continuous and positive on $[0,\infty)$.
Introduce the deterministic flow as follows: for $x \in [0,\infty)$
and $0\leq s \leq t$, 
\begin{equation*}
\varphi_{s, t } (x) = e^{ - \lambda (t-s) } x + \int_s^t e^{ - \lambda ( t-u) } a_u du.
\end{equation*}
For $t>0$ fixed and $y\in [0,\varphi_{0,t}(0)]$, let $\beta_t(y) \in [0,t]$ be uniquely determined
by $\varphi_{\beta_t(y),t}(0)=y$. For $y \geq \varphi_{0,t}(0)$, let $\gamma_t(y)
= (y-\varphi_{0,t}(0))e^{\lambda t}$, which satisfies $\varphi_{0,t}(\gamma_t(y))=y$.
It holds that for any $t>0$,
\begin{align*}
g(t,dy)=&\frac{p_{\beta_t(y)}}{a_{\beta_t(y)}}
\exp\Big(\int_{\beta_t(y)}^t (\lambda - f(\varphi_{\beta_t(y),s}(0)))ds \Big) \indiq_{\{y\in[0,\varphi_{0,t}(0))\}}dy\\
&+ \exp\Big(-\intot f(\varphi_{0,s}(\gamma_t(y)))ds   \Big) 
\indiq_{\{y\in [\varphi_{0,t}(0),\infty)\}} (g_0 \circ \gamma_t^{-1}) (dy).
\end{align*}
In particular, since $\beta_t(0)=t$, the density of $g(t)$ at $0$ is given by $g(t, 0) = p_t/a_t$.
\end{theo}

\subsection{Plan of the paper}
Section \ref{sec:2} consists of collecting some 
useful a priori bounds for the particle system and the limit process. 
In Section \ref{sec:3}, we check the path-wise uniqueness of the limit
process. Section \ref{sec:4} is devoted to the proof of Theorem \ref{theo:2},
the propagation of chaos without rate of convergence.
Section \ref{sec:5} shows the quantified propagation of chaos result. In Section \ref{sec:6},
we investigate the possible invariant distributions of the limit process.
The shape of the time-marginals is studied in Section \ref{sec:7},
in which we also prove the non-extinction result (Proposition \ref{paszero}) and the trend to equilibrium
when $\lambda=0$ (Proposition \ref{formal}).

\vip

\subsection{Constants} In the whole paper, $C$ stands for a (large) finite constant
and $c$ stands for a (small) positive constant. Their values may change from line to line.
They are 
allowed to depend only on $f,\lambda$ and $g_0$, any other dependence will be 
indicated in subscript. For example, $C_T$ is a finite constant depending only on
$f,\lambda,g_0$ and $T$.

\section{A priori bounds}\label{sec:2}
The aim of this section is to establish some fundamental  bounds for the particle system 
and for the limit process. Before that, we proceed to some elementary considerations.

\begin{rem}\label{rk1} Grant Assumptions \ref{ass:1} and \ref{ass:2} and additionally
Assumption  \ref{ass:3}
for point (iv).

\vip
(i) There is $c>0$ such that $f(x) \geq c x$ for all $x \geq 1$.

\vip

(ii) For all $A>0$, there is $C_A>0$ such that for all $x\geq 0$, $f(x+A)\leq C_A(1+f(x))$.

\vip

(iii) There is $C>0$ such that $f(x)\leq C\exp(Cx)$ for all $x\geq 0$.

\vip

(iv) For all $A>0$, there is $C_A>0$,  
$f(A(x+y)) \leq C_A(1+f(x)+f(y))$ for all $x,y\geq 0$.
\end{rem}

\begin{proof}
Point (i) is obvious since $f$ is convex increasing and since $f(0)=0$. 
Point (iii) is easily checked using that $f$ is increasing as well as point (ii) with $A=1$.
Point (iv) is an immediate consequence of Assumption \ref{ass:3}.
To check (ii), we will prove that there is $a>0$ such that $f(x+a) \leq 2f(x)+1$ for all $x\geq 0$,
which clearly suffices.
By Assumption \ref{ass:2}, 
there is $B>0$ such that $f'(x)\leq B(1+f(x))$ for all $x\geq 0$. Fix $a:=1/(2B)$ and write
$f(x+a)= f(x)+\int_x^{x+a} f'(y)dy \leq f(x)+ a B (1+\sup_{[x,x+a]}f) = f(x)+ a B (1+f(x+a)) =f(x)+(1+f(x+a))/2$,
whence $f(x+a) \leq 2f(x)+1$ as desired.
\end{proof}

We now study the limit equation.

\begin{prop} \label{estap}
Grant Assumption \ref{ass:1}, suppose that $\lambda \geq 0$ and that $\E[Y_0]<\infty$.
There is a constant $C>0$ depending only on $\lambda$, $f$ and $\E[Y_0]$ such that a
solution $(Y_t)_{t\geq 0}$ to \eqref{eq:dynlimit}
{\rm a priori} satisfies
\begin{align}
&\hbox{a.s., for all $t\geq 0$,} \quad Y_t \leq Y_0+C(1+t),\label{ttl1}\\
&\hbox{for all $t\geq 0$,}  \intot \E[Y_s f(Y_s)] ds \leq 2\E[Y_0]+ 2f(2)t. \label{ttl2}
\end{align}
\end{prop}

\begin{proof}
Taking expectations in \eqref{eq:dynlimit}, we see that  
\begin{align*}
\E [ Y_{ t }] = \E [Y_0 ]  +  \int_0^{t } \E[f(Y_s) ( 1 - Y_s  )]   ds
\le \E [Y_0 ] +  f( 2)  t  
- \frac 12\int_0^{t }  \E[Y_sf( Y_s)] ds,
\end{align*}
because $f(x)(1-x)= -xf(x)/2 + f(x)(1-x/2) \leq -xf(x)/2+f(2)$ for $x \geq 0,$ since $f$ is nonnegative
and non-decreasing. Inequality \eqref{ttl2} then follows from the fact that $\E[Y_t]\geq 0$.
Recalling the nonlinear SDE \eqref{eq:dynlimit}, we see that
$Y_t \leq Y_0 + \intot \E[\lambda Y_s + f(Y_s)] ds$ for all $t\ge 0$ a.s.
But there exists a constant $C$, depending on $f$ and $\lambda$, such that 
$\lambda y + f(y) \leq C(1+yf(y))$ for all $y\geq 0$: 
it suffices to use that $f$ is positive and non-decreasing. Consequently,
$\E[\lambda Y_s + f(Y_s)] \leq C (1+\E[Y_sf(Y_s)])$ for all $s\geq 0$ and \eqref{ttl1}
follows from \eqref{ttl2}.
\end{proof}

We now turn to the particle system.

\begin{prop}\label{prop:2}
Grant Assumption \ref{ass:1} and suppose that $\lambda \geq 0$.
Any solution $(X^N_t)_{t\geq 0}$ to \eqref{eq:dyn} a.s. satisfies that 
for all $t\geq 0$, all $i=1,\dots,N$,
\begin{align}
&\displaystyle X^{N,i}_t \leq  X^{N,i}_0 + (4\lambda t + 4)(\bar X^N_0 +Z^N_t), \label{tt1}\\
&\displaystyle \frac{1}{N} \sum_{j = 1 }^N 
\int_0^t \int_0^\infty (1+X^{N,j}_{s-}) \indiq_{\{ z \le f ( X^{ N, j }_{s-}) \}} \bN^j (ds, dz )
\le 3 \bar X^N_0 + 4 Z_t^N, \label{tt2}
\end{align}
where $Z^N_t:= N^{-1} \sum_{j = 1 }^N \int_0^t \int_0^\infty 
\indiq_{\{ z \le f(2) \}} \bN^j ( ds, dz ).$
Furthermore, it holds that for any $T\geq 0$,
\begin{align}\label{tt3}
\P\Big(\forall \; i=1,\dots,N,\;\;
\sup_{[0,T]}  X^{N, i } _t  \le   X^{N,i}_0 + (4\lambda T + 4)(\bar X^N_0 +2f(2)T)\Big) 
\geq 1 - e^{- NTf(2)(3-e)}.
\end{align}
\end{prop}

\begin{proof} We start with the following observation: taking the (empirical) mean of
\eqref{eq:dyn}, we find
\begin{equation}\label{eq:16}
 \bar X_t^N = \bar X_0^N + \frac{1}{N} \sum_{ i = 1 }^N \int_0^t \int_0^\infty 
\Big( \frac{N-1}{N} - X^{N, i }_{s- } \Big) \indiq_{\{ z \le f ( X^{ N, i }_{s-}) \}} 
\bN^i (ds, dz ) 
\end{equation}
which implies, since $\bar X_t^N\geq 0$ and $(N-1)/N\leq 1$, that 
$$
\frac{1}{N} \sum_{ i = 1 }^N \int_0^t \int_0^\infty ( X^{N, i }_{s- } - 1  ) 
\indiq_{\{ z \le f ( X^{ N, i }_{s-}) \}} \bN^i (ds, dz ) \le \bar X_0^N .
$$
Using that $x-1\geq (x+1)/3 - (4/3)\indiq_{ \{x\le 2\} }$ for all $x \geq 0$ and that $f$ is non-decreasing,
we deduce that 
\begin{align*}
&\frac{1}{N} \sum_{ i = 1 }^N \int_0^t \int_0^\infty  
(1+X^{N,i}_{s-}) \indiq_{\{ z \le f ( X^{ N, i }_{s-}) \}} \bN^i (ds, dz )  \\
\le&  3 \bar X_0^N 
+  \frac{4}{N} \sum_{ i = 1 }^N \int_0^t \int_0^\infty  \indiq_{ \{ X^{N, i }_{s- }  \le  2 \} }   
\indiq_{\{ z \le f ( X^{ N, i }_{s-}) \}} \bN^i (ds, dz ).
\end{align*}
Since $f$ is non-decreasing, \eqref{tt2} follows.
Recalling \eqref{eq:16} and using \eqref{tt2}, we realize that
$$ 
\bar X_t^N \le \bar X^N_0 + \frac{1}{N} \sum_{ i = 1 }^N \int_0^t \int_0^\infty 
\indiq_{\{ z \le f ( X^{ N, i }_{s-}) \}} 
\bN^i (ds, dz )  \leq  4 \bar X_0^N  + 4 Z_t^N .
$$ 
Now, for all $ 1 \le i \le N$, starting from \eqref{eq:dyn},
\begin{align*}
 X_t^{N, i } \le &X_0^{N, i } + \lambda \int_0^t \bar X_s^N ds + \frac1N \sum_{ j = 1 }^N 
\int_0^t \int_0^\infty \indiq_{ \{ z \le f ( X_{s-}^{N, j } ) \}} \bN^j ( ds, dz ) \\
\le &X_0^{N, i } +  \lambda \int_0^t ( 4\bar X_0^N  + 4 Z_s^N ) ds 
+ 3 \bar X^N_0 + 4 Z_t^N.
\end{align*}
Hence \eqref{tt1} follows from the fact that $Z^N_t$ is a.s. a non-decreasing function of time. 
Finally, the deviation estimate \eqref{tt3} simply relies on \eqref{tt1} and the inequality
\begin{align}\label{poiss}
\P\big(Z^N_T \geq 2f(2)T\big)\leq  e^{-2f(2)NT}\E\big[e^{NZ^N_T}\big]=e^{-NTf(2)(3-e)},
\end{align}
which uses that $Z^N_T$ is the empirical mean of $N$ i.i.d. Poisson$(f(2)T)$-random variables.
\end{proof}

The above estimate is largely sufficient to give the

\begin{proof}[Proof of Proposition \ref{theo:1}] 
Suppose first that $f$ is bounded. Then using only that $f$ is measurable (and nonnegative),
we can apply Theorem 9.1 in Chapter IV of Ikeda and Watanabe \cite{IW}: there is a path-wise unique solution 
$(X^{N}_t)_{t\geq 0}$ to \eqref{eq:dyn} defined on $[0, \infty)$. 

For a general $f$ satisfying Assumption \ref{ass:1}
and a fixed truncation level $K > 0$, we consider the unique solution $(X^{N, K}_t)_{t\geq 0}$ to 
\eqref{eq:dyn} with $f$ replaced by $f\land K$ and 
we introduce $\tau_K = \inf\{ t\geq 0 \, : \, |X^{N, K}_t | \geq K \}$. By path-wise uniqueness, 
it holds that
$X_t^{N, K} = X_t^{N, K+1}$ for all $K \in \N$ and all $ t \in [0,\tau_K]$
and that $\tau_K < \tau_{K+1}$  for all $K \in \N$,
almost surely. Setting $\tau= \sup_K \tau_K$, it is not hard to conclude that there is a path-wise unique
solution $(X^{N}_t)_{t\in [0,\tau)}$ to \eqref{eq:dyn} defined on $[0,\tau)$
and that $\limsup_{t \to \tau} |X^N_t|=\infty$ on the event $\{\tau <\infty\}$.

Recall now \eqref{tt1}: a.s., 
$X^{N,i}_t \leq  X^{N,i}_0 + C(1+t)(\bar X^N_0 +Z^N_t)$ for all $i=1,\dots,N$, all $t\geq 0$.
Observe also that obviously, $\sup_{[0,T]} Z^N_t <\infty$ a.s. for all $T>0$. Hence 
$\tau=\infty$ a.s., which completes the proof.
\end{proof}

\section{Path-wise uniqueness for the nonlinear SDE}\label{sec:3}

Let us first consider the case with compact support.

\begin{prop}\label{unicomp}
Suppose Assumption \ref{ass:1} and that $\lambda\geq 0$.
Path-wise uniqueness holds true for the nonlinear SDE \eqref{eq:dynlimit},
in the class of processes $(Y_t)_{t\geq 0}$ such that there is a deterministic locally 
bounded function $A:\R_+\mapsto\R_+$ such that a.s., $\sup_{t\geq 0} (Y_t/A(t)) <\infty$.
\end{prop}

Note that the above condition is {\it a priori} satisfied for $g_0$
compactly supported thanks to \eqref{ttl1}.

\begin{proof}
Consider two solutions 
$(Y_t)_{t\geq 0}$ and $(X_t)_{t\geq 0}$ to \eqref{eq:dynlimit}, driven by
the same Poisson measure $\bN$ and with $Y_0=X_0$. 
A very rough computation shows that 
there is a constant $C$, depending only on $\lambda$, such that
\begin{align}\label{ase}
\E[|X_t-Y_t|]\leq &
C \intot \E\Big[|X_s-Y_s|(1+f(X_s)+f(Y_s))+|f(X_s)-f(Y_s)|(1+X_s+Y_s)\Big] ds \\
&+ C \intot \Big(|\E[X_s]-\E[Y_s]| + |\E[f(X_s)]- \E[f(Y_s)]| \Big) ds. \nonumber
\end{align}
But we know by assumption that a.s., $\max\{Y_t,X_t\} \leq A(t)$ 
for some deterministic locally bounded function $A$.
Since $f$ is $C^1$ on $[0,\infty)$, it is Lipschitz continuous and bounded on compacts.
We thus easily check that for all $T$, there is a constant $C_T$ such that for all
$t\in [0,T]$,
\begin{align*}
\E[|X_t-Y_t|]\leq & C_T \intot \E[|X_s-Y_s|]ds.
\end{align*}
Finally, we know by assumption 
that the function $t\mapsto \E [|X_t -Y_t|]$ is locally bounded.
We thus may apply the Gronwall Lemma and deduce that $\E [|X_t - Y_t|]=0$ for all $t\geq 0$
as desired.
\end{proof}

\begin{prop}\label{unipascomp}
Let $\lambda\geq 0$ and grant Assumptions \ref{ass:1} and \ref{ass:2}.
Path-wise uniqueness holds true for the nonlinear SDE \eqref{eq:dynlimit}
in the class of processes $(Y_t)_{t\geq 0}$ such that $\sup_{[0,T]}\E[f(Y_t)]<\infty$
for all $T\geq 0$.

\vip

More generally, for any pair of solutions $(X_t)_{t\geq 0}$ and $(Y_t)_{t\geq 0}$ to \eqref{eq:dynlimit},
satisfying $\sup_{[0,T]}(\E[f(X_t)]+\E[f(Y_t)])<\infty$ for all $T\geq 0$,
driven by the same Poisson measure but with possibly different initial conditions, it holds that 
for all $T\geq 0$,
\begin{equation}\label{stab}
\sup_{[0,T]} \E[|H(X_t)-H(Y_t)|] \leq C_T \E[|H(X_0)-H(Y_0)|],
\end{equation}
where $H(x)=f(x)+\arctan x$.
\end{prop}

Here again, \eqref{ttl1} (and Remark \ref{rk1}-(ii)) shows that the condition is {\it a priori} 
satisfied if $\E[f(Y_0)]<\infty$.
As already mentioned, a proof based on $\E[|X_t-Y_t|]$ does not seem to work:
one finds an inequality like \eqref{ase} (even with a finer computation using the It\^ o formula),
from which it seems difficult to conclude.

\vip

The rest of the section is devoted to the proof of Proposition \ref{unipascomp}.

\begin{lem}\label{lem:1}
Grant Assumptions \ref{ass:1} and \ref{ass:2} and let $H(x) = f(x) + \arctan (x)$. 
There is a constant $C$ such that for all $x,y \in \R_+$, we have 

\vip

(0) $|H''(x)|\leq C H'(x)$, 

\vip

(i) $x+ H' (x) \le C ( 1 + f(x) )$,

\vip

(ii) $|x-y| + |H'(x) - H' (y) | + |f(x) - f(y) | \le C | H(x) - H(y) |$,

\vip

(iii) $- \sg(x-y) ( x H'(x) - y H' (y) ) \le C | H(x) - H(y) |$,

\vip

(iv) $- (f (x) \land f(y)) | H (x) - H(y) | + | f(x) - f(y) | 
( H (x) \wedge H (y)  - | H(x) - H(y ) | ) \le  C |H (x) - H(y) |$.
\end{lem}

\begin{proof}
First, $|H''(x)| \leq |\arctan''(x)|+f''(x)\leq C+f''(x)$.
If $x\leq 1$, we deduce that $|H''(x)|\leq C \leq C H'(x)$, while if $x\geq 1$, we recall
that $f''(x)\leq Cf'(x)$, whence $|H''(x)|\leq C(1+f'(x)) \leq C f'(x)\leq C H'(x)$ as desired.

\vip

We next check (i). We have 
$ x + H' (x) \le x + f'(x) + 1$. If $x\leq 1$, we just write 
$ x + H' (x) \le C \leq C(1+f(x))$. If now $x\geq 1$, 
since $f'(x)\leq Cf(x)$ by Assumption \ref{ass:2}, 
we find that $ x + H' (x) \le 2x + Cf(x)\leq C f(x)$ by Remark \ref{rk1}-(i).

\vip

In order to prove (ii), it is sufficient to check that 
$ 1 + | H''(x)| + f'(x) \le C H'(x)$ for all $x\geq 0$.
First, $1\leq CH'(x)$ for all $x\geq 0$, because $H'(x)\geq f'(1)>0$
if $x\geq 1$, while $H'(x) \geq \arctan'(x)\geq 1/2$ if $x\leq 1$. Next, $f'(x)\leq
H'(x)$ is obvious. Finally, $|H''(x)| \leq  C H'(x)$ has already been seen.

\vip

Concerning point (iii), 
\begin{align*}
  - \sg( x-y) ( x H'(x) - y H' (y) )  
=  - \sg ( x-y) ( x f'(x) - y f' (y) ) 
  - \sg ( x-y) \Big( \frac{x}{1 +x^2 } - \frac{y}{1 +y^2 } \Big) .
\end{align*}
The first term on the RHS is negative, because $ x f' (x) $ is non-decreasing. 
The second one can be roughly bounded by $C |x-y|$
which in turn is bounded by $C | H(x) - H(y) |$ due to point (ii).

\vip

Finally, we rewrite the LHS of point (iv) as
$$
- (f(x) \vee f(y) ) | H (x) - H(y) | + (H(x) \wedge H(y)) | f(x) - f(y) |.
$$
Using that $ | f(x) - f(y) |\leq | H(x) - H(y) |$ because $H(x)=f(x)+\arctan(x)$ with both $f$ and $\arctan$ 
non-decreasing, that $f(x) \vee f(y)\geq f(x)$ and $H(x) \wedge H(y) \leq H(x)$, we get an upper-bound in 
$$
- f(x) | H (x) - H(y) |  + H(x) | H(x) - H(y) | = (\arctan x) | H(x) - H(y) | \leq \pi| H(x) - H(y) |/2.
$$
This completes the proof.
\end{proof}

\begin{proof}[Proof of Proposition \ref{unipascomp}]
Let thus $(X_t)_{t\geq 0}$ and $(Y_t)_{t\geq 0}$ be two solutions of 
\eqref{eq:dynlimit} driven by the same Poisson measure and satisfying $\sup_{[0,T]}\E[f(X_t)+f(Y_t)]<\infty$
for all $T>0$.
We apply the It\^o formula for jump processes and take expectations
to compute $\E [| H ( X_t) - H(Y_t) |]$. Actually, one has to first consider a regularized
version of the absolute value and then to pass to the limit, but this causes no 
difficulty. See the proof of Theorem \ref{theo:4} where such a regularization procedure is handled.
We find, using that $H$ is increasing (whence $\sg(H(x)-H(y))=\sg(x-y)$)
and that $H(0)=0$,
\begin{align*}
&\E[| H ( X_t) - H(Y_t) |]=\E[| H ( X_0) - H(Y_0) |] -\lambda \intot \E\Big[\sg(X_s-Y_s) (H'(X_s)X_s-H'(Y_s)Y_s)  
\Big]ds\\
&+\lambda \intot \E\Big[\sg(X_s-Y_s)(H'(X_s)\E[X_s]-H'(Y_s)\E[Y_s] )  
\Big]ds\\
&+\!\intot\! \E \Big[\! -\!  (f(X_s) \wedge f(Y_s)) | H(X_s)- H(Y_s) |
\!+\!   |f(X_s) - f(Y_s) | ( H(X_s) \wedge H(Y_s) - |H(X_s) - H(Y_s) | ) \Big]ds \\
&+ \intot \E\Big[\sg(X_s-Y_s)(H'(X_s)\E[f(X_s)]-H'(Y_s)\E[f(Y_s)] )\Big]ds\\
&=:\E[| H ( X_0) - H(Y_0) |]+I_t+J_t+K_t+L_t.
\end{align*}
First, it follows from Lemma \ref{lem:1}-(iii)-(iv) that
$I_t+K_t \leq C \intot \E [| H(X_s) - H(Y_s) |]ds $. Next we write
\begin{align*}
J_t+L_t \leq & \intot \E[| H'(X_s)- H'(Y_s) |](\lambda \E[X_s]+\E[f(X_s)]) ds \\
&+ \intot \E[H'(Y_s)](\lambda |\E[X_s-Y_s]|+ |\E[f(X_s)-f(Y_s)|]) ds .
\end{align*}
Using Lemma \ref{lem:1}-(i)-(ii), we thus find
$$
J_t+L_t \leq  C \intot \E[| H(X_s)- H(Y_s) |] \E[1+f(X_s)+f(Y_s)] ds.
$$
Since $\sup_{[0,T]}\E[f(X_t)+f(Y_t)]<\infty$ by assumption, we conclude that
for all $T\geq 0$, there is a constant $C_T$ such that for all $t\in[0,T]$,
$$
\E [| H ( X_t) - H(Y_t) |]\leq \E[| H ( X_0) - H(Y_0) |]+C_T \intot \E [| H(X_s) - H(Y_s) |]ds.
$$
We know by assumption 
that $t\mapsto \E [| H ( X_t) - H(Y_t) |]$ is locally bounded.
Hence \eqref{stab} follows from the Gronwall Lemma. Path-wise uniqueness is immediately deduced
by injectivity of $H$.
\end{proof}

\section{Propagation of chaos without rate}\label{sec:4}

In this section, we prove Theorem \ref{theo:6} and conclude the proof
of Theorem \ref{theo:2}. We start with tightness.

\begin{proof}[Proof of Theorem \ref{theo:6}-(i)-(ii)]
First, it is well-known that point (ii) follows from point (i) and the exchangeability 
of the system, see Sznitman \cite[Proposition 2.2-(ii)]{s}. We thus
only prove (i). We consider a probability distribution $g_0$
on $\R_+$ such that $\int_0^\infty x g_0(dx)<\infty$ and, for each $N\geq 1$,
the unique solution $(X^N_t)_{t\geq 0}$ to \eqref{eq:dyn} starting from
some i.i.d. $g_0$-distributed initial conditions $X^{N,i}_0$.
To show that the family $((X^{N,1}_t)_{t\geq 0})_{N\geq 1}$ is tight $\D(\R_+)$, 
we use the criterion of Aldous, see Jacod and Shiryaev \cite[Theorem 4.5 page 356]{js}. It is sufficient to prove that

\vip

(a) for all $ T> 0$, all $\e>0$,
$ \lim_{ \delta \downarrow 0} \limsup_{N \to \infty } \sup_{ (S,S') \in A_{\delta,T}} 
\P ( |X_{S'}^{N, 1 } - X_S^{N , 1 } | > \varepsilon ) = 0$,
where $A_{\delta,T}$ is the set of all pairs of stopping times $(S,S')$ such that
$0\leq S \leq S'\leq S+\delta\leq T$ a.s.,

\vip

(b) for all $ T> 0$, $\lim_{ K \uparrow \infty } \sup_N 
\P ( \sup_{ t \in [0, T ] } X_t^{N, 1 } \geq K ) = 0$.

\vip

Point (b) follows from \eqref{tt1}: we know that $\sup_{[0,T]} X^{N,1}_t
\leq X^{N,1}_0+ (4\lambda T + 4)(\bar X^N_0+Z^N_T)$, with $Z^N_T$ the mean of $N$ i.i.d. 
Poisson$(Tf(2))$-distributed random variables. Hence, setting $m=\int_0^\infty xg_0(dx)$,
$$
\sup_N \E\Big[\sup_{[0,T]} X^{N,1}_t \Big] \leq m+ (4\lambda T + 4)(m+f(2)T)<\infty.
$$

To check (a), we will use several times that for $0\leq t \leq T$, all $N\geq 1$, all $j=1,\dots,N$,
\begin{align} \label{ethop}
\intot \E[X^{N,j}_s f(X^{N,j}_s)] ds \leq 3m+4f(2)t \leq C_T.
\end{align}
Indeed, take expectations in \eqref{tt2}, use
exchangeability and recall that $\E[\bar X^N_0]=m$ and $\E[Z^N_t]=tf(2)$.

\vip

We next consider $(S,S')\in A_{\delta,T}$ and write
\begin{align*}
| X_{S'}^{N, 1 } - X_S^{N , 1 }| \le&   
\int_S^{S'} \int_0^\infty X^{ N, 1 }_{u- } \indiq_{\{ z \le f ( X_{u- }^{N, 1} ) \}} \bN^1 (du, dz )
+ \frac1N \sum_{j=2}^N \int_S^{S'} \int_0^\infty \indiq_{\{ z \le f ( X_{u- }^{N, j} ) \}} 
\bN^j (du, dz )  \\
&+ \lambda \int_S^{S'} X_u^{N, 1} du + \lambda \int_S^{S'} \bar X^N_u   du\\
 =:& I_{S, S'} + J_{S, S'} + K_{S, S'} +L_{S,S'}.
\end{align*}
We first note that $I_{S,S'}>0$ implies that 
$\tilde I_{S,S'}:=
\int_S^{S'} \int_0^\infty \indiq_{\{ z \le f ( X_{u- }^{N, 1} ) \}} \bN^i (du, dz)\geq 1$, whence
\begin{align*}
\P ( I_{S, S'} > 0 )\leq \P (\tilde I_{S,S'}\geq 1)\leq \E[I_{S,S'}]\le 
\E\Big[ \int_S^{S+\delta} f( X_u^{N, 1 } ) du\Big].
\end{align*}
But for all $A>0$, $f(x) \leq f(A)+ A^{-1} x f(x)$ because $f$ is non-decreasing.
Hence, by \eqref{ethop}, 
\begin{align*}
\P ( I_{S, S'} > 0 )\le \delta f(A) 
+ \frac1A \int_0^T \E [ X_u^{N, 1 } f( X_u^{N, 1 } )] du \le  \delta f(A) + \frac{C_T}{A}.
\end{align*}
Choosing $A=f^{-1}(\delta^{-1/2})$ (recall that $\lim_{\infty}f=\infty$ and consider
a generalized notion of inverse function if necessary), we end with
\begin{align*}
\P ( I_{S, S'} > 0 )\le \delta^{1/2} + \frac{C_T}{f^{- 1 } (\delta^{-1/2} ) }.
\end{align*}
We proceed similarly to check that
\begin{align*}
\E[J_{S,S'}]\leq & \frac1N \sum_{j=2}^N\E\Big[ \int_S^{S+\delta} f( X_u^{N, j} ) du\Big]
\le \delta^{1/2} + \frac{C_T}{f^{- 1 } ( \delta^{-1/2} ) }.
\end{align*}
Next, we write, for any $A>0$, using that $x \leq A + xf(x)/f(A)$ and then \eqref{ethop},
\begin{align*}
\E [ K_{S, S'} ] \le \lambda \E \Big[\int_S^{S+\delta } X_u^{N, 1} du \Big]
\leq \lambda\delta A + \frac{\lambda}{f(A)}\int_0^T \E[X_u^{N, 1 } f( X_u^{N, 1 } )] du
\leq \lambda \delta A + \frac{\lambda C_T}{f(A)}.
\end{align*}
We choose $A=\delta^{-1/2}$ and get 
\begin{align*}
\E [ K_{S, S'} ] \le \lambda \delta^{1/2} + \frac{\lambda C_T}{f(\delta^{-1/2})}.
\end{align*}
The same arguments show that
\begin{align*}
\E [ L_{S, S'} ] \le \lambda \delta^{1/2} + \frac{\lambda C_T}{f(\delta^{-1/2})}.
\end{align*}
We can now conclude that for $\e>0$,
\begin{align*}
\P(| X_{S'}^{N, 1 } - X_S^{N , 1 }|>\e) \leq & \P(I_{S,S'}>0)+\P(J_{S,S'}>\e/4)+\P(K_{S,S'}>\e/4)
+\P(L_{S,S'}>\e/4)\\
\leq& \delta^{1/2} + \frac{C_T}{f^{- 1 } (\delta^{-1/2}) } + \frac4\e
\Big((1+2\lambda) \delta^{1/2} + \frac{C_T}{f^{- 1 } (\delta^{-1/2} ) } + 
\frac{2\lambda C_T}{f(\delta^{-1/2})} \Big).
\end{align*}
This last quantity does not depend on $N\geq 1$ nor on $(S,S')\in A_{\delta,T}$ and tends
to $0$ as $\delta\to 0$. This completes the proof.
\end{proof}

We now turn to the consistency result.

\begin{proof}[Proof of Theorem \ref{theo:6}-(iii)]
We consider a probability distribution $g_0$
on $\R_+$ such that $\int_0^\infty x g_0(dx)<\infty$ and, for each $N\geq 1$,
the unique solution $(X^N_t)_{t\geq 0}$ to \eqref{eq:dyn} starting from
some i.i.d. $g_0$-distributed initial conditions $X^{N,i}_0$.
We introduce $\mu_N=N^{-1}\sum_{i=1}^N \delta_{(X^{N,i}_t)_{t\geq 0}}$, which is a 
$\cP(\D(\R_+))$-valued random variable. By Theorem \ref{theo:6}-(ii), 
this sequence is tight. We thus consider a (not relabeled) subsequence $\mu_N$ going
in law to some $\cP(\D(\R_+))$-valued random variable $\mu$. We want to show that
$\mu$ a.s.\ belongs to $\cS:=\{\cL((Y_t)_{t\geq 0}) \, : \, (Y_t)_{t\geq 0}$
solution to \eqref{eq:dynlimit} with $\cL(Y_0)=g_0$ and satisfying $\int_0^t \E[Y_sf(Y_s)]ds <\infty$ 
for all $t\geq 0\}$.

\vip

{\it Step 1.} For $t\geq 0$, we introduce $\pi_t:\D(\R_+)\mapsto \R_+$ 
defined by $\pi_t(\gamma)=\gamma_t$. We claim that 
$Q\in\cP(\D(\R_+))$ belongs to $\cS$ if the following conditions are satisfied:

\vip
(a) $Q\circ \pi_0^{-1}=g_0$;

\vip
(b) for all $t\geq 0$, $\int_{\D(\R_+)}\intot \gamma_sf(\gamma_s) ds Q(d\gamma)<\infty$;

\vip
(c) for any $ 0 \le s_1 < \ldots < s_k < s < t$, any $\varphi_1,\dots,\varphi_k 
\in C_b ( \R_+)$, any $\varphi \in C^2_b (\R_+)$,
\begin{align*}
F(Q):=&\int_{{\mathcal D} ( \R_+ )}   \int_{{\mathcal D}  (\R_+ )}   
Q ( d \gamma ) Q ( d \tilde \gamma ) \; \varphi_1 ( \gamma_{s_1}  ) \ldots  
\varphi_{k} (\gamma_{s_k} ) \\
&\Big[ \varphi ( \gamma_t) - \varphi ( \gamma_s) - \int_s^t f( \gamma_u) 
( \varphi ( 0) - \varphi (\gamma_u ) ) du - \int_s^t \varphi' ( \gamma_u ) 
[ f (\tilde \gamma_u ) + \lambda ( \tilde \gamma_u - \gamma_u ) ] du\Big]=0 .
\end{align*}

Indeed, consider a process $(Y_t)_{t\geq 0}$ of which the law $Q$ satisfies the
above three points. By (a), $\cL(Y_0)=g_0$. By (b),
$\int_0^t \E[Y_sf(Y_s)]ds <\infty$ for all $t\geq 0$. Finally, (c) implies that
for all $\varphi \in C^2_b(\R_+)$,
$$
\varphi(Y_t)- \intot [\varphi(0)-\varphi(Y_s)]f(Y_s)ds - \intot \varphi'(Y_s)\Big[ \E[f(Y_s)]+
\lambda(\E[Y_s]-Y_s) \Big]ds
$$
is a martingale. By Jacod and Shiryaev \cite[Theorem II.2.42 page 86]{js}, this implies that $Y$ 
is a semimartingale with characteristics 
$ (B, C, \nu) $ given by 
$$ B_t = \int_0^t \Big[\lambda
(\E[Y_s]-Y_s) + \E[f(Y_s)] +  Y_s f(Y_s) \Big] ds ,\;  C_t=0 , \; \nu (ds, dx) =  f( Y_{s-}) ds
\delta_{ - Y_{s-}} (dx) . $$
We have chosen the truncation function $ h(x) = x$ (i.e.\ {\it no truncation}) 
since $Y$ possesses only large jumps. Finally, 
\cite[Theorem III. 2.26 page 157]{js} implies that there is a Poisson measure $\bN(ds,dz)$ on $\R_+\times\R_+$
with intensity $dsdz$ such that $Y$ solves \eqref{eq:dynlimit}.

\vip

{\it Step 2.} Here we check that for any $t\geq 0$, a.s., $\mu(\{\gamma \, : \, \Delta\gamma(t)\ne 0\})=0$.
We assume by contradiction that there exists $t > 0 $ 
such that $\mu ( \{ \gamma  : \Delta \gamma (t) \neq 0 \} ) > 0 $
with positive probability. Hence there are $a,b>0$ such that the event 
$E:=\{\mu ( \{ \gamma : |\Delta \gamma (t) | > a  \} ) > b\}$ has a 
positive probability. For every $\varepsilon > 0$, we have
$E\subset \{ \mu (  \cB^\varepsilon_a  ) > b\}$, where 
$\cB^\varepsilon_a := \{ \gamma : \sup_{ s \in (t- \varepsilon , t + \varepsilon)}| \Delta \gamma ( s) | > a \}$, which
is an open subset of $\D ( \R_+ )$. Thus $\cP^{\varepsilon}_{a,b} := 
\{  Q \in {\cP} ( {\D} ( \R_+) ) : Q (  \cB^\varepsilon_a  ) > b \}$
is an open subset of $ {\cP} ( {\D} ( \R_+) )$. 
The Portmanteau theorem implies then that for any $\e>0$,
$$ 
\liminf_{N \to \infty } \P ( \mu_N \in \cP^{\varepsilon}_{a,b}  ) \geq \P ( \mu \in \cP^{\varepsilon}_{a,b}  ) 
\geq \P ( E)  > 0.
$$
But, for all $N>1/a$ (so that for each $i=1,\dots, N$, the only jumps of $X^{N,i}$ that may exceed $a$ are those 
produced by the Poisson measure $\bN^i$),
\begin{align*}
\{\mu_N \in \cP^{\varepsilon}_{a,b}\} \subset& \Big\{\frac1N \sum_{ i= 1 }^N  
\indiq_{\{ \int_{t- \varepsilon}^{t + \varepsilon}  \indiq_{ \{ z \le f( X^{N, i }_{u- } ) \}} \bN^i (du, dz) \geq 1\}} \geq b\Big\},
\end{align*}
whence, using exchangeability,
\begin{align*}
\P ( \mu_N \in \cP^{\varepsilon}_{a,b}  ) 
\le & \frac{1}{b N} \sum_{i=1}^N \E \Big(\int_{t- \varepsilon}^{t + \varepsilon}  \indiq_{ \{ z \le f( X^{N, i }_{u- } ) \}} \bN^i (du, dz) 
 \Big)=   \frac{1}{b} \int_{t- \varepsilon}^{t + \varepsilon} \E [ f ( X_u^{ N, 1} ) ] du.
\end{align*}
We now observe that for any $A>0$,
$\int_{t- \varepsilon}^{t + \varepsilon} \E [ f ( X_u^{ N, 1} ) ] du \leq 2\varepsilon f(A) + A^{-1}
\int_{t- \varepsilon}^{t + \varepsilon} \E [ X_u^{N,1}f ( X_u^{ N, 1} ) ] du \leq 2\varepsilon f(A) + C A^{-1}$ by
\eqref{ethop}. Choosing $A=f^{-1}(\e^{-1/2})$, we conclude that 
$\int_{t- \varepsilon}^{t + \varepsilon} \E [ f ( X_u^{ N, 1} ) ] du \leq C(\sqrt\varepsilon + 1/f^{-1}(\e^{-1/2}))$,
which does not depend on $N$ and tends to $0$ as $\e\to 0$. We thus have the contradiction
$$ 
0 < \P ( E) \le \liminf_{\varepsilon \to 0 } \liminf_{N \to \infty } \P ( \mu_N  \in \cP^{\varepsilon}_{a, b}) =0.
$$

{\it Step 3.} Our limit $\mu$ a.s. satisfies (a), because $\mu \circ \pi_0^{-1}$ 
is the limit in law of $\mu^N \circ \pi_0^{-1}=N^{-1}\sum_{i=1}^N \delta_{X^{N,i}_0}$,
which goes to $g_0$ because the $X^{N,i}_0$ are i.i.d. with common law
$g_0$. It also a.s. satisfies (b) since for all $t\geq 0$ and $K > 0,$ using the Fatou Lemma 
and \eqref{ethop},
\begin{align*}
\E\Big[\int_{\D(\R_+)}\intot [(\gamma_sf(\gamma_s) )\wedge K ]ds \mu(d\gamma)\Big]
\leq& \liminf_N \E\Big[\int_{\D(\R_+)}\intot [(\gamma_sf(\gamma_s))\wedge K] ds \mu_N(d\gamma)\Big]\\
= & \liminf_N N^{-1} \sum_{i=1}^N\intot \E[X^{N,i}_sf(X^{N,i}_s)]ds < \infty.
\end{align*}
The conclusion follows by letting $K \to \infty .$ 

\vip

{\it Step 4.} It remains to check that $\mu$ a.s. satisfies (c). We thus consider
$F:\D(\R_+)\mapsto \R$ as in (c).

\vip

{\it Step 4.1.} Here we prove that $\lim_N\E[|F(\mu_N)|]=0$.
We have 
\begin{align*}
F( \mu_N) =& \frac1N \sum_{i= 1}^N \varphi_1 ( X^{N, i }_{s_1} ) 
\ldots \varphi_k ( X^{N, i }_{s_k} ) \\
&\Bigg[ \varphi (X^{N, i }_{t})  -  \varphi (X^{N, i }_{s})  - \int_s^t f(  X^{N, i }_{u}) 
[\varphi (0) -   \varphi (X^{N, i }_{u}) ] du 
- \lambda \int_s^t  \varphi' (X^{N, i }_{u})  (\bar X^N_u -  X^{N, i }_{u}) du\\
& \hskip5cm - \int_s^t \varphi' (X^{N, i }_{u}) \frac1N \sum_{j=1}^N f ( X_u^{ N, j})  du 
\Bigg] .
\end{align*}
But recalling \eqref{eq:dyn} and using the It\^o formula for jump processes,
\begin{align*}
\varphi ( X_t^{N, i } ) =& \varphi (X_0^{N, i } ) + 
\int_0^t \! \int_0^\infty \!
[ \varphi ( 0 ) - \varphi( X^{N, i }_{u-} ) ] \indiq_{ \{ z \le f( X^{N, i }_{u- } ) \}} 
\bN^{i} (du, dz) 
+ \lambda \intot \varphi'( X^{N, i }_u)(\bar X^N_u -X^{N,i}_u)du \\
&+ \sum_{ j \neq i } \int_0^t \int_0^\infty \Big( \varphi ( X^{N, i }_{ u - } + \frac1N ) 
- \varphi ( X_{u-}^{N, i } ) \Big) \indiq_{ \{ z \le f( X_{u-}^{N, j } ) \}} \bN^j (du, dz).
\end{align*}
Consequently, using the notation $\tilde \bN^i (du, dz ) = \bN^i (du, dz ) - du dz$ and setting
\begin{align*}
M_t^{N, i }  := & \int_0^t \int_0^\infty
[ \varphi ( 0 ) - \varphi( X^{N, i }_{u-} ) ] \indiq_{ \{ z \le f( X^{N, i }_{u- } ) \}} 
\tilde \bN^{i} (du, dz),\\
\Delta_t^{N, i } :=& \sum_{ j \neq i } \! \int_0^t \! \int_0^\infty \!\!\!  \Big( \varphi ( X^{N, i }_{ u - } + 
\frac1N ) - \varphi ( X_{u-}^{N, i } ) \Big) \indiq_{ \{ z \le f( X_{u-}^{N, j } ) \}}\bN^j (du, dz) 
\!-\!  \int_0^t \! \varphi' (X^{N, i }_{u}) \frac1N \sum_{j=1}^N \! f ( X_u^{ N, j} ) du,
\end{align*}
we see that 
$$ 
F(\mu_N)  =  \frac1N \sum_{i= 1}^N \varphi_1 ( X^{N, i }_{s_1} ) \ldots 
\varphi_k ( X^{N, i }_{s_k} ) \big[ ( M_t^{N, i } - M_s^{N, i } ) 
+ ( \Delta_t^{N, i } - \Delta_s^{N, i } ) \big] .
$$ 
Since the Poisson measures $\bN^i$ are i.i.d., the martingales $M^{N, i }$
are orthogonal.
Using exchangeability and the boundedness of the $\varphi_k$, we thus find that
\begin{equation}\label{eq:318}
\E [ |F ( \mu_N) | ] \le C_F  \frac{1}{\sqrt{N}} 
\E [ ( M_t^{N, 1} - M_s^{N, 1 } )^2]^{1/2}  + C_F \E[  | \Delta_t^{N, 1}| +|\Delta_s^{N, 1 }|]. 
\end{equation} 
First, since $\varphi$ is bounded and using \eqref{ethop} (recall that 
$f(x)\leq f(1)+xf(x)$), 
$$
\E[( M_t^{N, 1} - M_s^{N, 1 } )^2]=\int_s^t \E[(\varphi ( 0 ) - \varphi( X^{N, 1}_{u} ))^2
f( X^{N, 1}_{u} )]  du \leq C _F \intot \E[f( X^{N, 1}_{u} )]  du \leq C_F. 
$$
Next,
\begin{align*}
|  \Delta_t^{N, 1 }| \le & \int_0^t \int_0^\infty \Big|\varphi ( X^{N, 1 }_{ u - } + \frac1N ) - 
\varphi ( X_{u-}^{N, 1 } )\Big| \indiq_{ \{ z \le f( X_{u-}^{N, 1 } ) \}}\bN^1 (du, dz) \\
& + \Big| \sum_{j=1 }^N \int_0^t \int_0^\infty  \big( \varphi ( X^{N, 1 }_{ u - } + \frac1N ) - 
\varphi ( X_{u-}^{N, 1 } ) \big) \indiq_{ \{ z \le f( X_{u-}^{N, j} ) \}}\tilde \bN^j (du, dz)\Big| 
\\
& + \sum_{j=1}^N \int_0^t \Big| \varphi ( X^{N, 1 }_{ u} + \frac1N ) - 
\varphi ( X_{u}^{N, 1 } ) - \frac1N \varphi' (X_u^{N, 1 } )\Big| f( X_u^{N, j} ) du \\
& =: I^N_t+J^N_t+K^N_t.
\end{align*}
Using that $\varphi'$ is bounded and \eqref{ethop}, we find
$$ 
\E [ I^N_t ] \le \frac{C_F}{N} \int_0^t \E[f ( X_u^{N, 1})] du 
\le \frac{C_F}{N}.
$$ 
Moreover, since $\varphi''$ is bounded and by \eqref{ethop} again,
$$ 
\E [ K^N_t] \le \frac{C_F}{N^2}\sum_{j=1}^N \int_0^t \E[f ( X_u^{N, j} )] du \le \frac{C_F}{N}.
$$
Finally, using the independence of the Poisson measures $\bN^j$, that 
$\varphi'$ is bounded and \eqref{ethop},
$$ 
\E [(J^N_t)^2] = \sum_{j=1}^N \intot \E\Big[\big( \varphi ( X^{N, 1 }_{ u} + \frac1N ) - 
\varphi ( X_{u}^{N, 1 } ) \big)^2 f( X_{u}^{N, j} )\Big] du 
\le \frac{C_F}{N^2 } \sum_{ j = 1}^N \int_0^t \E [ f( X_u^{N, j } ) ] du \le \frac{C_F}{N}.
$$
All this implies that $\E [ |  \Delta_t^{N, 1 }| ] \le C_F/\sqrt{N}$ whence, coming 
back to \eqref{eq:318}, $\E [ |F ( \mu_N) | ]  \le C_F/ \sqrt{N}$.

\vip

{\it Step 4.2.} Clearly, $F$ is continuous at any point $Q\in \cP(\D(\R_+))$
such that $Q(\gamma\, : \, \Delta\gamma(s_1)=\dots=\Delta\gamma(s_k)=\Delta\gamma(s)
=\Delta\gamma(t)=0)=1$ and such that $\int_{\D(\R_+)}\intot [\gamma_u+f(\gamma_u)]du Q(d\gamma)<\infty$.
Our limit point $\mu$ a.s. satisfies these two conditions by Steps 2 and 3 (because
$x+f(x)\leq C(1+xf(x))$).
Since $\mu$ is the limit in law of $\mu_N$ and since $F$ is a.s. continuous at $\mu$,
we thus deduce that for any $K>0$, $\E[|F(\mu)|\land K]=\lim_N \E[|F(\mu_N)|\land K]$.
Consequently, $\E[|F(\mu)|\land K] \leq \limsup_N \E[|F(\mu_N)|]$ for all $K>0$. 
Using Step 4.1, we deduce that $\E[|F(\mu)|\land K] =0$ for any $K>0$.
By the monotone convergence theorem, we conclude that $\E[|F(\mu)|]=0$, whence
$F(\mu)=0$ a.s.
\end{proof}

We can finally study the well-posedness of the nonlinear SDE.

\begin{proof}[Proof of Theorem \ref{theo:2}]
Point (i) (weak existence assuming only Assumption \ref{ass:1} and that $\E[Y_0]<\infty$) follows
from Theorem \ref{theo:6}-(ii)-(iii): we have built at least one weak solution,
passing to the limit in the particle system, and we have seen that this solution
satisfies that $\int_0^t \E[Y_s f(Y_s)]ds <\infty$ for all $t\geq 0$.

\vip

For point (ii) (strong well-posedness under Assumption \ref{ass:1} 
when $g_0=\cL(Y_0)$ is
compactly supported), we only have to check that the solution built in point (i)
satisfies that there is a deterministic locally bounded function $A:\R_+\mapsto\R_+$
such that a.s., $Y_t\leq A(t)$ for all $t\geq 0$. This will conclude the proof,
since such a weak existence result, together with the path-wise uniqueness
proven in Proposition \ref{unicomp}, will imply the strong well-posedness.
We thus assume that Supp $g_0\subset [0,K]$ and set $A(t):=K+\intot (\lambda \E[Y_s]+\E[f(Y_s)])ds$,
which is clearly locally bounded since  $\int_0^t \E[Y_s f(Y_s)]ds <\infty$ for all $t\geq 0$.
Then it is obvious, recalling \eqref{eq:dynlimit}, that a.s., for all $t\geq 0$,
$Y_t\leq A(t)$. 

\vip

To check point (iii) (strong well-posedness under Assumptions \ref{ass:1} and \ref{ass:2}
when $\E[f(Y_0)]<\infty$),
it suffices to prove that the solution built in point (i)
satisfies $\sup_{[0,t]} \E[f(Y_s)]<\infty$ for all $t\geq 0$. Again, this weak
existence, together with the strong uniqueness of Proposition \ref{unipascomp}, will
complete the proof. 
Put $C(t):=\intot (\lambda \E[Y_s]+\E[f(Y_s)])ds$, which is again locally bounded, and observe from
\eqref{eq:dynlimit}, that a.s., for all $t\geq 0$, $Y_t\leq Y_0+C(t)$. Since $\E[f(Y_0)]<\infty$, 
we immediately conclude, using Remark \ref{rk1}-(ii), 
that $\sup_{[0,t]} \E[f(Y_s)]<\infty$ for all $t\geq 0$, as desired.
\end{proof}

Finally, we can give the

\begin{proof}[Proof of Theorem \ref{theo:6}-(iv)] 
First grant Assumption \ref{ass:1} and assume that $g_0$ is compactly supported.
We have seen in Theorem \ref{theo:6}-(ii)-(iii) that $\mu_N$ is tight and that
any limit point $\mu$ a.s. belongs to $\cS=\{\cL((Y_t)_{t\geq 0}) \, : \, (Y_t)_{t\geq 0}$
solution to \eqref{eq:dynlimit} with $\cL(Y_0)=g_0$ and satisfying $\int_0^t \E[Y_sf(Y_s)]ds <\infty$ 
for all $t\geq 0\}$. But arguing as in the proof of Theorem \ref{theo:2}-(ii),
we see that $\cS=\cS'$, where $\cS'=\{\cL((Y_t)_{t\geq 0}) \, : \, (Y_t)_{t\geq 0}$
solution to \eqref{eq:dynlimit} with $\cL(Y_0)=g_0$ and such that a.s., for all $t\geq 0$, $Y_t\leq A(t)$
for some deterministic locally bounded function $A\}$. As seen in Theorem 
\ref{theo:2}-(ii),
$\cS'$ is reduced to one point. The conclusion follows: $\mu_N$ goes in probability,
as $N\to \infty$, to the unique element of $\cS'$.

\vip

Next grant Assumptions \ref{ass:1} and \ref{ass:2} and assume that 
$\int_0^\infty f(y)g_0(dy)<\infty$. 
We have seen in Theorem \ref{theo:6}-(ii)-(iii) that $\mu_N$ is tight and that
any limit point $\mu$ a.s. belongs to $\cS=\{\cL((Y_t)_{t\geq 0}) \, : \, (Y_t)_{t\geq 0}$
solution to \eqref{eq:dynlimit} satisfying $\int_0^t \E[Y_sf(Y_s)]ds <\infty$ 
for all $t\geq 0\}$. But arguing as in the proof of Theorem \ref{theo:2}-(ii),
we see that $\cS=\cS''$, where $\cS''=\{\cL((Y_t)_{t\geq 0}) \, : \, (Y_t)_{t\geq 0}$
solution to \eqref{eq:dynlimit} with $\cL(Y_0)=g_0$ and  such that $\sup_{[0,t]}\E[f(Y_s)]<\infty$ 
for all $t\geq 0\}$.
As seen in Theorem \ref{theo:2}-(iii),
$\cS''$ is reduced to one point. The conclusion follows.
\end{proof}

\section{Quantified propagation of chaos}\label{sec:5}

The aim of this section is to prove Theorem \ref{theo:4}. We thus impose
Assumptions \ref{ass:1}, \ref{ass:2} and \ref{ass:3} and we fix an initial distribution $g_0$ 
such that $ \int_0^\infty f^2 (x) g_0 (dx) < \infty$.
We consider an i.i.d.\ family $ X_0^{ N, i }$ of $g_0$-distributed random variables,
an i.i.d.\ family of Poisson measures $\bN^i(ds,dz)$ on $\R_+\times\R_+$ with intensity
measure $dsdz$, we denote, for each $N\geq 1$, by $(X^N_t)_{t\geq 0}=
(X^{N,1}_t,\dots,X^{N,N}_t)_{t\geq 0}$ the solution to \eqref{eq:dyn}. Finally, we denote by
$(Y^{N,i}_t)_{t\geq 0}$, for every $N\geq 1$, every $i =1,\dots, N$, the path-wise unique
(thanks to Theorem \ref{theo:2}-(iii)) solution to \eqref{eq:dynlimit}
starting from $X^{N,i}_0$ and driven by the Poisson measure $\bN^i$. Obviously,
for every $N\geq 1$, the processes $(Y^{N,i}_t)_{t\geq 0}$, $i=1,\dots,N$, are i.i.d.

\vip

To prove Theorem \ref{theo:4}, we will essentially mimic the path-wise uniqueness proof
of Theorem \ref{theo:2} to control $\sup_{[0,T]}\E[|H(X^{N,1}_t)-H(Y^{N,1}_t)|]$ by $C_T/\sqrt N$.
But there are a number of technical difficulties.
First, we need to work on $[0,\tau_N^T]$, for some well-chosen stopping time
$\tau_N^T$ that is asymptotically greater than $T$.
Next, we will rather study $\E[(N^{-1}+(H(X^{N,1}_t)-H(Y^{N,1}_t))^2)^{1/2}]$: this changes nothing to the result,
but allows for a more rigorous proof (we apply the It\^o formula to a true $C^2$ function)
and allows for the control of a second derivative, see Lemma \ref{tech}-(i), that would explode
without the additional $N^{-1}$ term.
We start with some more moment estimates.

\begin{lem}\label{lem:2}
(i) For all $T>0$, there is $C_T$ depending only on $T$, $\lambda$, $g_0$ and $f$ such that
$$ 
\E \Big[\sup_{[0,T]} f^2 ( Y_t^{N,1} ) \Big] \le C_T
\quad \hbox{and}\quad \sup_N \E \Big[\sup_{[0,T]}f^2 ( X_t^{ N, 1} )\Big] \le C_T.
$$

(ii) For all $T\geq 1$, we can find a constant $R_T>0$ such that the stopping time
$$ 
\tau_N^T := \inf \{ t\geq 0 : N^{-1} \sum_{i=1}^N (f ( X_t^{N, i} ) + f (Y_t^{N, i } ) ) \geq R_T \}
$$
satisfies, for some constants 
$C>0$ (and $C_T$) depending only on $\lambda$, $g_0$ and $f$ (and $T$).
$$ 
\P ( \tau_N^T \le T ) \le \frac{C}{N} \quad \hbox{and}\quad
\E \Big[\sup_{[0,T]}( 1 + f ( X_t^{ N, 1} ) + f (Y_t^{ N, 1 } ) ) \indiq_{ \{ \tau_N^T \le T \} } \Big]
\le \frac{C_T}{\sqrt{N} }.
$$
\end{lem}

\begin{proof}
Recalling \eqref{ttl1}, it a.s.\ holds that for all $t\geq 0$,
$Y_t^{N,1} \le X_0^{ N,1} + C(1+t)$. Using Remark \ref{rk1}-(ii) and that 
$\E[f^2(X_0^{N,1})]=\int_0^\infty f^2(x) g_0 (dx)<\infty$, we immediately deduce that
$\E [\sup_{[0,T]} f^2 ( Y_t^{N,1} )] \leq C_T$.

\vip

Next, \eqref{tt1} tells us that a.s., for all $t\geq 0$, 
$ X_t^{N,1} \le X_0^{N,1} + C(1+T) ( \bar X_0^N + Z_T^N )$. By Remark \ref{rk1}-(iv),
$$ 
\sup_{[0,T]} f^2  ( X_t^{N, 1}) \le C_T (1+ f^2 (X_0^{N, 1} )+ f^2 ( \bar X_0^N ) + f^2 ( Z_T^N) ).
$$
But $f^2 $ being convex, $ f^2 ( \bar X_0^N ) \le N^{-1} \sum_{i=1}^N f^2 ( X_0^{N, i} )$. Consequently,
$\E[\sup_{[0,T]} f^2(X^{N,1}_t)]\leq C_T(1+ \int_0^\infty f^2(x) g_0 (dx)+\E[f^2(Z_T^N)])$.
To end the proof of (i), it suffices to recall that $Z^N_T$ is the mean of $N$ i.i.d. Poisson$(Tf(2))$-distributed
random variables: since $f(x)\leq Ce^{Cx}$ by Remark \ref{rk1}-(iii), a simple computation shows that indeed,
$\sup_N\E[f^2(Z_T^N)]<\infty$.

\vip

Using again \eqref{ttl1} and \eqref{tt1}, we see that a.s., for all $t\in [0,T]$,
all $i=1,\dots,N$, $ X_t^{N,i} \le X_0^{N,i} + C(1+T) ( \bar X_0^N + Z_T^N )$ and $Y_t^{N,i} \le X_0^{ N,i} + C(1+t)$.
Consequently, using Remark \ref{rk1}-(iv) and the convexity of $f$
(whence $f(\bar X_0^N) \leq N^{-1}\sum_{i=1}^N f(X^{N,i}_0)$),
$$ 
\sup_{[0,T]} \frac1N \sum_{i=1}^N  (f ( X_t^{N, i} ) + f (Y_t^{N, i } ) ) \le C_T \Big( 1 + f (Z_T^N ) 
+ \frac1N \sum_{i=1}^N f ( X_0^{N, i } ) \Big) .
$$
The bounds $ \P ( Z_T^N \geq 2 f(2) T ) \le \exp(- N T f(2) (3 - e ) )$, see \eqref{poiss},
and
$$ \P \Big( \frac1N \sum_{i=1}^N f ( X_0^{N, i } ) \geq \int_0^\infty f(x) g_0(dx) + 1 \Big) 
\le \frac{ {\rm Var} ( f (X_0^{N, 1 } ) ) }{N } \le \frac{C}{ N} $$
imply that, with the choice $R_T=C_T(1+f(2f(2)T)+\int_0^\infty f(x) g_0(dx) + 1)$,
$$
\P(\tau_N^T\leq T) \leq  \exp(- N T f(2) (3 - e ) ) +C/N \leq C /N
$$ 
as desired.
The last inequality immediately follows, using (i) and the Cauchy-Schwarz inequality. 
\end{proof}

We carry on with a technical lemma similar to Lemma \ref{lem:1}.

\begin{lem}\label{tech}
Grant Assumptions \ref{ass:1} and \ref{ass:2} and recall that $H(x)=f(x)+\arctan x$. Define, for $N\geq 1$,
$a_N ( x, y) := [ N^{-1} + ( H(x) - H(y))^2]^{1/2}$.

\vip

(i) It holds that $|\partial_x a_N (x, y ) | \le H' (x)$ and 
$| \partial_{xx} a_N ( x, y ) | \le |H'' (x) | + \sqrt{N} (H' (x) )^2$.

\vip

(ii) We have $| \partial_x a_N (x, y ) + \partial_y a_N ( x, y ) | \le | H' (x) - H' (y) |$.

\vip

(iii) There is $C>0$ such that $-[x \partial_x a_N( x, y ) + y \partial_y a_N ( x, y )] \le C a_N ( x, y )$.

\vip

(iv) Finally, there is $C>0$ such that
\begin{align*}
\Delta_N(x,y):=& (f(x) \wedge f(y)) [ a_N ( 0,0 ) - a_N ( x, y ) ] + (f(x) - f(y) )_+ [a_N ( 0, y ) - a_N ( x, y) ] 
\\
&+ ( f(y) - f(x) )_+ [ a_N ( x, 0 ) - a_N ( x, y ) ] \\
\leq & \frac{f(x) \wedge f(y) }{\sqrt{N} } + C a_N ( x, y ).
\end{align*}
\end{lem}

\begin{proof}
Points (i) and (ii) follow from direct computations. For (iii), using the expression of $H$,
\begin{align*}
-[x \partial_x a_N( x, y ) + y \partial_y a_N ( x, y )]=& \frac{-(H(x)-H(y))}{ [ N^{-1} + ( H(x) - H(y))^2]^{1/2}}
[x H'(x) - yH'(y)]\\
=&\frac{-(H(x)-H(y))}{ [ N^{-1} + ( H(x) - H(y))^2]^{1/2}}
[x f'(x) - yf'(y)]\\
&+\frac{-(H(x)-H(y))}{ [ N^{-1} + ( H(x) - H(y))^2]^{1/2}}
\Big[\frac x{1+x^2} -\frac y{1+y^2}\Big].
\end{align*}
The first term on the RHS is non-positive, because both $H(x)$ and $xf'(x)$ are non-decreasing. 
The second one is roughly bounded
by $|x/(1+x^2)-y/(1+y^2)|\leq |x-y|$ which is bounded, recalling Lemma \ref{lem:1}-(ii), by
$C|H(x)-H(y)|\leq C a_N(x,y)$. To prove (iv), we first observe, since $a_N$ is symmetric
and $f$ is non-decreasing, that
\begin{align*}
 \Delta_N ( x, y) 
= \frac{f(x) \wedge f(y) }{\sqrt{N} } - ( f(x) \vee f(y)) a_N ( x, y )  + |f(x) - f(y) | a_N ( 0, x \wedge y ) .
\end{align*}
Noting that $  |f(x) - f(y) | \le  |H(x) - H(y) |\le a_N ( x, y )$, we deduce that
\begin{align*}
\Delta_N ( x, y) \le &\frac{f(x) \wedge f(y) }{\sqrt{N} } + a_N ( x, y ) ( a_N ( 0, x \wedge y ) - f(x) 
\vee f(y)).
\end{align*}
The conclusion follows, since $a_N(0,x\wedge y) - f(x) \vee f(y)\leq N^{-1/2} + H(x)\wedge H(y) - f(x) \vee f(y)$,
which is obviously bounded by $1+\pi/2$.
\end{proof}

We are now ready to give the

\begin{proof}[Proof of Theorem \ref{theo:4}] We fix $T>0$ and define $R_T$ 
and $\tau_N^T$ as in Lemma \ref{lem:2}-(ii).  In the whole proof, we work on the time interval $[0,T]$.
Recall that $a_N$ and $\Delta_N$ were defined in Lemma \ref{tech}.
\vip

{\it Step 1.} This is the main step of the proof. We show
that there is a constant $C_T$ such that for all $N\geq 1$,
$\sup_{[0,T]} \E [ a_N ( X^{N, 1 }_{t \wedge \tau_N^T}, Y^{ N, 1 }_{ t \wedge \tau_N^T } ) ] \leq C_T N^{-1/2}$.
Applying the It\^o formula for jump processes, we find that
$$
\E [ a_N ( X^{N, 1 }_{t \wedge \tau_N^T}, Y^{ N, 1 }_{ t \wedge \tau_N^T } ) ] = N^{-1/2}+I + J + \lambda K+ \lambda L , 
$$
where
\begin{align*}
I=&\E\Big[ \int_0^{t \wedge \tau_N^T } \Delta_N ( X^{N, 1 }_s, Y^{N, 1 }_s ) ds \Big],\\
J=&\sum_{j=2}^N \E \Big[\int_0^{t \wedge \tau_N^T }  f ( X^{N, j }_s) [ a_N ( X^{N, 1 }_s + \frac1N , Y^{N, 1 }_s) 
- a_N (  X^{N, 1 }_s, Y^{N, 1 }_s ) ]ds \Big]\\ 
&+ \E \Big[\int_0^{t \wedge \tau_N^T }  \partial_y a_N ( X^{N, 1 }_s, Y^{N, 1 }_s )  \E [ f( Y_s^{N, 1 } ) ] ds\Big], \\
K=&- \E \Big[ \int_0^{t \wedge \tau_N^T }  \big( \partial_x a_N ( X^{N, 1 }_s, Y^{N, 1 }_s ) X_s^{ N, 1} 
+ \partial_y a_N ( X^{N, 1 }_s, Y^{N, 1 }_s ) Y_s^{N, 1 }  \big) ds \Big],\\
L=& \E \Big[\int_0^{t \wedge \tau_N^T }  \big( \partial_x a_N ( X^{N, 1 }_s, Y^{N, 1 }_s ) \bar X_s^{ N} 
+ \partial_y a_N ( X^{N, 1 }_s, Y^{N, 1 }_s ) \E [Y_s^{N, 1 }]   \big) ds\Big].
\end{align*} 

By Lemma \ref{tech}-(iv) and Lemma \ref{lem:2}-(i),
$$
I \le \frac 1{\sqrt{N}}\intot \E[f(Y^{N,1}_s)] ds 
+ C \E\Big[ \int_0^{t\land \tau_N^T} a_N ( X^{N, 1 }_{s}, Y^{N, 1 }_{s} ) \Big] ds
\leq  \frac {C_T}{\sqrt{N}}
+ C \int_0^t \E [a_N ( X^{N, 1 }_{s \wedge \tau_N^T}, Y^{N, 1 }_{s \wedge \tau_N^T} ) ] ds.
$$

Lemma \ref{tech}-(iii) implies that
$$ 
K \le C \E\Big[ \int_0^{t\land \tau_N^T} a_N ( X^{N, 1 }_{s}, Y^{N, 1 }_{s} ) \Big] ds
\leq C \int_0^t \E [a_N ( X^{N, 1 }_{s \wedge \tau_N^T}, Y^{N, 1 }_{s \wedge \tau_N^T} ) ] ds.
$$

We next write $L =L_1 + L_2 + L_3 $, with 
\begin{align*}
L_1 = & \E\Big[ \int_0^{t \wedge \tau_N^T }   \partial_x a_N ( X^{N, 1 }_s, Y^{N, 1 }_s ) [ \bar X_s^{ N} - \bar Y_s^{ N}] ds 
\Big], \\
L_2 = & \E \Big[\int_0^{t \wedge \tau_N^T }  [ \partial_x a_N ( X^{N, 1 }_s, Y^{N, 1 }_s ) 
+ \partial_y a_N ( X^{N, 1 }_s, Y^{N, 1 }_s )] \bar Y_s^{ N} ds\Big] , \\
L_3 =& \E \Big[\int_0^{t \wedge \tau_N^T } \partial_y a_N ( X^{N, 1 }_s, Y^{N, 1 }_s ) ( \E [ Y_s^{ N, 1 } ] - \bar Y_s^{ N} )
ds\Big].
\end{align*}
Using the Cauchy-Schwarz inequality, Lemma \ref{tech}-(i) and the fact that
the $Y^{N,i}_s$ are i.i.d., 
$$ 
L_3 \le \frac1 {\sqrt N} \int_0^t \E [ H' ( Y^{N, 1 }_s  )^2 ]^{1/2} ({\rm Var}\; Y_s^{ N, 1 })^{1/2} ds 
\leq \frac{C_T}{\sqrt N}.
$$
The last inequality follows from Lemma \ref{lem:1}-(i), which tells us that 
$ x+H' (x) \le C(1+ f(x))$, whence
$\sup_{[0,T]} \E [ H' ( Y^{N, 1 }_s  )^2 ] \leq C_T$ and $\sup_{[0,T]}{\rm Var}\; Y_s^{ N, 1 } \leq C_T$
by Lemma \ref{lem:2}-(i).
Next, Lemmas \ref{tech}-(ii) and \ref{lem:1}-(ii) tell us that 
$|\partial_x a_N (x,y)+ \partial_y a_N (x,y)|\leq |H'(x)-H'(y)|\leq C|H(x)-H(y)|\leq
C a_N(x,y)$. Consequently,
$$ 
L_2 \le C \E \Big[\int_0^{t \wedge \tau_N^T }  |\bar Y_s^{ N}|  a_N(X^{N, 1 }_s,Y^{N, 1 }_s) ds \Big] \leq
C_T \int_0^t \E [a_N ( X^{N, 1 }_{s \wedge \tau_N^T}, Y^{N, 1 }_{s \wedge \tau_N^T} ) ] ds.
$$
We used that, by definition of $\tau_N^T$ and since $x\leq C(1+f(x))$ (see Lemma \ref{lem:1}-(i)), 
$|\bar Y_s^{ N}|\leq C(1+ N^{-1}\sum_{i=1}^N f(X^{N,i}_s)) \leq C(1+R_T)$ for all $s\in [0,\tau_N^T]$ a.s.
Finally, using that $\tau_N^T$ does not break the exchangeability and Lemma \ref{tech}-(i), we write
\begin{align*}
L_1  =& \frac 1N \sum_{j=1}^N\E \Big[\int_0^{t \wedge \tau_N^T }   \partial_xa_N(X_s^{N, 1},Y^{N,1}_s) 
[X_s^{N,j} - Y_s^{N,j}] ds \Big] \\
=&\frac 1N \sum_{j=1}^N\E \Big[\int_0^{t \wedge \tau_N^T }   \partial_xa_N(X_s^{N,j},Y^{N,j}_s) 
[X_s^{N,1} - Y_s^{N,1}] ds \Big]\\
\leq & \E \Big[\int_0^{t \wedge \tau_N^T }   \Big(\frac 1 N \sum_{j=1}^N H'(X_s^{N,j})\Big)|X_s^{N,1} - Y_s^{N,1}| ds \Big]\\
\leq & C_T  \int_0^t \E [a_N ( X^{N, 1 }_{s \wedge \tau_N^T}, Y^{N, 1 }_{s \wedge \tau_N^T} ) ] ds.
\end{align*} 
The last inequality uses that, by definition of $\tau_N^T$ and since $H'(x)\leq C(1+f(x))$ 
(see Lemma \ref{lem:1}-(i)), $|N^{-1}\sum_{j=1}^N H' (X_s^{N, j })|\leq C_T$ for all $s\in [0,\tau_N^T]$ a.s. It
also uses that $|x-y| \leq C|H(x)-H(y)|\leq Ca_N(x,y)$ by Lemma \ref{lem:1}-(ii).

\vip 

We finally write $ J = J_1 + J_2 + J_3 + J_4 $, where, using again exchangeability,
\begin{align*}
J_1 =&  \E  \Big[\int_0^{t \wedge \tau_N^T }   f ( X^{N, 2 }_s) \Big( (N-1) [ a_N ( X^{N, 1 }_s + \frac1N , Y^{N, 1 }_s) 
- a_N (  X^{N, 1 }_s, Y^{N, 1 }_s )] - \partial_x a_N ( X_s^{ N, 1 } , Y_s^{N, 1 } ) \Big)  ds\Big],\\
J_2 = &\E \big[ \int_0^{t \wedge \tau_N^T }  f ( X^{N, 2 }_s) ( \partial_x a_N ( X_s^{ N, 1 } , Y_s^{N, 1 } ) 
+\partial_y a_N ( X_s^{ N, 1 } , Y_s^{N, 1 } )) ds \Big],\\
J_3 =& \E \Big[\int_0^{t \wedge \tau_N^T }  \partial_y a_N ( X_s^{ N, 1 } , Y_s^{N, 1 } ) 
[ f ( Y^{N, 2 }_s) -f ( X^{N, 2 }_s) ] ds \Big], \\
J_4 =&  \E \Big[\int_0^{t \wedge \tau_N^T }  \partial_y a_N ( X_s^{ N, 1 } , Y_s^{N, 1 } ) 
[ \E [ f ( Y^{N, 2 }_s)]  -f ( Y^{N, 2 }_s) ] ds \Big].
\end{align*}
We start with $J_1$. Using Lemma \ref{tech}-(i),
\begin{align*}
&|(N-1)[a_N(x+1/N,y)-a_N(x,y)] -\partial_x a_N(x,y)|\\
\leq & |a_N(x+1/N,y)-a_N(x,y)|+ |N[a_N(x+1/N,y)-a_N(x,y)] -\partial_x a_N(x,y)| \\
\leq & N^{-1} \sup_{z\in [x,x+1/N]}[ |\partial_x a_N(z,y)| +|\partial_{xx} a_N(z,y)| ]\\
\leq & N^{-1} \sup_{z\in [x,x+1/N]}[ H'(z) + |H''(z)|+ \sqrt N (H'(z))^2 ] \\
\leq & C N^{-1/2} (1+f^2(x)).
\end{align*}
The last inequality uses that $|H''(x)|\leq C H'(x)$ (see Lemma \ref{lem:1}),
the fact that $H'(x) \leq C(1+f(x))$ (see Lemma \ref{lem:1}-(i)) and that
$\sup_{[x,x+1/N]} f(z) \leq C(1+f(x))$ (see Remark \ref{rk1}-(iv)). Consequently,
$$
J_1 \leq \frac C {\sqrt N} \E\Big[\int_0^{t \wedge \tau_N^T }   f ( X^{N, 2 }_s)(1+f^2(X^{N,1}_s)) ds \Big].
$$
By Lemmas \ref{tech}-(ii) and \ref{lem:1}-(ii),
$| \partial_x a_N (x, y ) + \partial_y a_N ( x, y ) | \le | H' (x) - H' (y) | \leq C|H(x)-H(y)|\leq C a_N(x,y)$. 
Hence
$$
J_2 \leq  C \E\Big[\int_0^{t \wedge \tau_N^T }   f ( X^{N, 2 }_s) a_N(X^{N,1}_s,Y^{N,1}_s) \Big].
$$
Lemmas \ref{tech}-(i) and \ref{lem:1}-(i) imply that $|\partial_y a_N (x,y)| \leq H'(y) \leq C(1+f(y))$ 
and we obviously have $|f(x)-f(y)|\leq |H(x)-H(y)|\leq a_N(x,y)$. It follows that
$$
J_3 \leq  C \E\Big[\int_0^{t \wedge \tau_N^T }   (1+f ( Y^{N,1}_s)) a_N(X^{N,2}_s,Y^{N,2}_s) \Big]
= C \E\Big[\int_0^{t \wedge \tau_N^T }   (1+f ( Y^{N,2}_s)) a_N(X^{N,1}_s,Y^{N,1}_s) \Big].
$$
We have checked that
$$
J_1+J_2+J_3 \leq  C \E\Big[\int_0^{t \wedge \tau_N^T }   (1+f(X^{N,2}_s)+f ( Y^{N,2}_s)) \Big(a_N(X^{N,1}_s,Y^{N,1}_s) 
+\frac{1+f^2(X^{N,1}_s)}{\sqrt N}\Big)ds\Big].
$$
Using exchangeability and then the definition of $\tau_N^T$, we thus can write
\begin{align*}
J_1+J_2+J_3 \leq&  C \E\Big[\int_0^{t \wedge \tau_N^T }  \Big(\frac 1 N \sum_{j=1}^N (1+f(X^{N,j}_s)+f ( Y^{N,j}_s))\Big) 
\Big(a_N(X^{N,1}_s,Y^{N,1}_s) 
+\frac{1+f^2(X^{N,1}_s)}{\sqrt N}\Big)ds\Big]\\
\leq & C(1+R_T) \E\Big[\int_0^{t \wedge \tau_N^T } \Big(a_N(X^{N,1}_s,Y^{N,1}_s) 
+\frac{1+f^2(X^{N,1}_s)}{\sqrt N}\Big)ds\Big]\\
\leq & C_T \intot \E[a_N( X^{N, 1 }_{s \wedge \tau_N^T}, Y^{N, 1 }_{s \wedge \tau_N^T} )] ds + \frac {C_T }{\sqrt N}.
\end{align*}
The last inequality uses that $\sup_N \sup_{[0,T]} \E[f^2(X^{N,1}_t)] <\infty$ by Lemma \ref{lem:2}-(i).
Finally, using again exchangeability, 
that $|\partial_y a_N (x,y)| \leq C(1+f(y))$, the Cauchy-Schwarz inequality and that the
$Y^{N,i}_s$ are i.i.d.,
\begin{align*}
J_4 = & \E\Big[ \int_0^{t\wedge \tau_N^T} \partial_y a_N(X^{N,1}_s,Y^{N,1}_s) \Big(\E [ f ( Y^{N, 2 }_s)]
- \frac 1{N-1} \sum_{j=2}^N f(Y^{N,j}_s) \Big)ds\Big] \\
\leq & C \intot  \E[(1+f(Y^{N,1}_s))^2]^{1/2} \frac{[ {\rm Var} f(Y^{N, 1 }_s)]^{1/2}}{\sqrt{N-1}} ds.
\end{align*}
Again, we conclude that $J_4 \leq C_T N^{-1/2}$ since 
$\sup_{[0,T]} \E [f^2 ( Y_t^{ N, 1 } )] < C_T$, as shown in Lemma \ref{lem:2}.

\vip

All in all, we have checked that  $\E[a_N( X^{N, 1 }_{t \wedge \tau_N^T}, Y^{N, 1 }_{t \wedge \tau_N^T} )]
\leq C_TN^{-1/2} + C_T \intot  \E[a_N( X^{N, 1 }_{s \wedge \tau_N^T}, Y^{N, 1 }_{s \wedge \tau_N^T} )]ds$.
We conclude the step with the help of the Gronwall Lemma.

\vip

{\it Step 2.} It is not hard to complete the proof. First, gathering Step 1 
(recall that $|H(x)-H(y)|\leq a_N(x,y)$)
and Lemma \ref{lem:2}-(ii) (recall that $H(x)\leq \pi/2+f(x)$), we find, for all $t\in[0,T]$,
$$
\E[|H(X^{N,1}_t)-H(Y^{N,1}_t)|] \leq \E[|H(X^{N,1}_{t\land \tau_N^T})-H(Y^{N,1}_{t\land \tau_N^T})|] 
+ \E[(H(X^{N,1}_t)+H(Y^{N,1}_t))\indiq_{\{\tau_N^T\leq T\}}]\leq \frac{C_T}{\sqrt N}.
$$
Moreover, $|x-y|\leq C |H(x)-H(y)|$ by Lemma \ref{lem:1}-(ii), whence 
$\sup_{[0,T]}\E[|X^{N,1}_t-Y^{N,1}_t|] \leq C_T N^{-1/2}$.

\vip

We next assume additionally that $\int_0^\infty y^{2+\e} g_0(dy)<\infty$ for some $\e>0$. Recalling 
\eqref{ttl1}, this obviously implies that $\sup_{[0,T]} \E[(Y_t^{N,1})^{2+\e}] \leq C_T$.
Since the $Y^{N,i}_t$ are i.i.d.\ $\R$-valued random variables, it is well-known, see e.g. 
\cite[Theorem 1 with $d=1,p=1,q=2+\e$]{fg}, that 
$$
\E\Big[\cW_1\Big(N^{-1}\sum_{i=1}^N \delta_{Y^{N,i}_t},\cL(Y_t^{N,1})\Big)\Big] \leq 
\frac{C \E[(Y_t^{N,1})^{2+\e}]^{1/(2+\e)}}{\sqrt N}
\leq \frac{C_T}{\sqrt N}.
$$
But it follows from exchangeability that 
$$
\E\Big[ \cW_1\Big(N^{-1}\sum_{i=1}^N \delta_{X^{N,i}_t},
N^{-1}\sum_{i=1}^N \delta_{Y^{N,i}_t}\Big)\Big] \leq \frac 1 N \sum_{i=1}^N \E[|X^{N,i}_t-Y^{N,i}_t|]
= \E[|X^{N,1}_t-Y^{N,1}_t|]\leq \frac{C_T}{\sqrt N}.
$$
Using the triangular inequality for $\cW_1$, we conclude that for all $t\in [0,T]$,
$$
\E\Big[\cW_1\Big(N^{-1}\sum_{i=1}^N \delta_{X^{N,i}_t},\cL(Y^{N,1}_t)\Big)\Big] \leq \frac{C_T}{\sqrt N}
$$
as desired.
\end{proof}

\section{Invariant distributions}\label{sec:6}

Here we prove Theorem \ref{theo:41}.
We thus only impose Assumption \ref{ass:1}.
We start with the following remark.

\begin{prop}\label{SDElin}
Let $\bN$ be a Poisson measure on $\R_+\times\R_+$ with intensity $dsdz$,  let $\lambda \geq 0$ and $a\geq 0$.

\vip

(i) The $\R_+$-valued SDE
\begin{equation}\label{La}
Z_t=Z_0 - \intot \int_0^\infty Z_{s-} \indiq_{\{z\leq f(Z_{s-})\}}\bN(ds,dz) + \intot (a-\lambda Z_s)ds
\end{equation}
has a path-wise unique solution for every nonnegative initial condition $Z_0$.

\vip

(ii) Furthermore, \eqref{La} has a unique invariant probability measure $g_a$. It is given by $g_0=\delta_0$
if $a=0$ and by $g_a(dx)=g_a(x)dx$ if $a>0$, where (with the convention that $a/\lambda=\infty$ if
$\lambda=0$),
$$ 
g_a (x) = \frac{p_a}{a - \lambda x } \exp \Big( - \int_0^x \frac{f(y) }{a - \lambda y} dy \Big) 
\indiq_{\{ 0 \le y < a/\lambda \} },$$
where  $p_a>0$ is such that 
$ \int_0^\infty g_a(x) dx = 1$. It automatically holds that $\int_0^\infty f(x)g_a(dx)=p_a$.
\end{prop}

\begin{proof}
Point (i) is straightforward. All the coefficients being locally Lipschitz-continuous, we have {\it local} 
strong existence and uniqueness, i.e. strong existence and uniqueness on $[0,\tau)$, where $\tau=\inf\{
t\geq 0\, : \, Z_t = \infty\}$. But having a look at \eqref{La}, we see that a.s., for all $t\geq 0$,
$Z_t \leq Z_0+at$. Hence $\tau=\infty$ a.s.

\vip

Point (ii) is straightforward if $a=0$. Indeed, $\delta_0$ is clearly an invariant distribution.
It is unique, because for any initial condition, $Z_t$ tends a.s. to $0$ as $t\to \infty$.
Indeed, if $\lambda>0$, then $0 \leq Z_t \leq e^{-\lambda t} Z_0$.
If now $\lambda=0$, then $Z_t=Z_0\indiq_{\{t < \tau_0\}}$, where $\tau_0$ follows an exponential 
distribution with parameter $f(Z_0)$ (conditionally on $Z_0$).

\vip

We next prove (ii) when $a>0$.
We first claim that the homogeneous Markov process $Z$ has exactly one invariant 
probability distribution which is supported in $[0,a/\lambda]$ (or $[0,\infty)$ if $\lambda=0$).
This follows from the classical theory of Markov processes, 
since $0$ is a positive Harris recurrent state of $Z.$ Indeed, let
$\tau_0=\inf\{ t\geq 0 \, : \, Z_t=0\}$. Then for any initial 
condition $z>0$, $\E_z(\tau_0)<\infty$. This can be easily checked, using e.g. that starting from $z>0$,
$Z_t=e^{-\lambda t}z + (1-e^{-\lambda t})a/\lambda \geq \min\{z,a/\lambda\}$ for all $t\in [0,\tau_0)$, so that
$Z$ jumps to zero with a rate bounded from below by $\min\{f(z),f(a/\lambda)\}>0$. As a consequence, 
the successive jump times of $Z$ to $0$ induce a regeneration scheme, and $Z$ is positive Harris 
recurrent implying the uniqueness of the invariant probability measure. Moreover, it is clear that
$Z_t \le a/\lambda$ for every $t\geq \tau_0$, which implies that the support of the invariant 
probability is included in $[0, a/\lambda ] .$

\vip

It thus only remains to check that $g_a$ is indeed an invariant probability measure for \eqref{La}.
The computations below include the case where $\lambda=0$.
It suffices to prove that for all $\phi \in C^1_b(\R_+)$,
\begin{equation}\label{obj}
\int_0^\infty [\phi(0)-\phi(x)]f(x)g_a(dx)+\int_0^\infty \phi'(x)(a-\lambda x)g_a(dx)=0.
\end{equation}
Indeed, the infinitesimal generator associated to the SDE \eqref{La} is given by
$\cL_a \phi (x) = [\phi(0)-\phi(x)]f(x)+\phi'(x)(a-\lambda x)$.
First, a direct computation shows that
$$
\int_0^\infty\!\! f(x)g_a(dx)=p_a \int_0^{a/\lambda}\!\! \frac{f(x)}{a - \lambda x }  
\exp \Big( - \int_0^x \frac{f(y) }{a - \lambda y} dy \Big)dx=- p_a
\Big[ \exp \Big( - \int_0^x \frac{f(y) }{a - \lambda y} dy \Big)\Big]_{x=0}^{x=a/\lambda}=p_a.
$$
The last equality uses that $f(a/\lambda)>0$. Hence, \eqref{obj} reduces to
$$
\int_0^{a/\lambda} \phi(x) f(x) g_a(x)dx - \int_0^{a/\lambda} \phi'(x)(a-\lambda x)g_a(x)dx= \phi(0)p_a.
$$
Proceeding to an integration by  parts in the second integral and using that $f(x) g_a(x) 
+ [(a-\lambda x)g_a(x)]'=0$ for all $x \in (0,a/\lambda)$, we see that \eqref{obj} reduces to
$$
- \Big[\phi(x)(a-\lambda x)g_a(x) \Big]_{x=0}^{x=a/\lambda} = \phi(0)p_a.
$$
This is easily checked, since $ag_a(0)=p_a$ and since $\lim_{x\uparrow a/\lambda}(a-\lambda x)g_a(x)=0$.
\end{proof}

We now study for which values of $a$ an invariant measure of \eqref{La}
is an invariant measure of \eqref{eq:dynlimit}.

\begin{lem}\label{unimi}
Adopt the notation of Proposition \ref{SDElin}.
When $a=0$, we define $p_0=0=\int_0^\infty f(x)g_0(dx)$.
We also introduce, for $a\geq 0$, $m_a:=\int_0^\infty x g_a(dx)$. 
The equation $a=\lambda m_a +p_a$ has the solution $a=0$
and exactly one positive solution $a_*$. Furthermore, it holds that $a_* > \lambda$.
\end{lem}

\begin{proof}
The proof below works whenever $\lambda >0$ or $\lambda=0$.
Evidently, $a=0$ solves $a=\lambda m_a +p_a$.
Let now $a>0$. Since $ \int_0^\infty g_a(dx) = 1$, we have
$$ 
\frac{1}{p_a} = \int_0^{ a / \lambda } \frac{1}{ a - \lambda x } 
\exp \Big(- \int_0^x \frac{f(y) }{a - \lambda y } dy  \Big) dx =: \Gamma_1 (a) .
$$
Next, 
$$ 
m_a=p_a \int_0^{a/\lambda } \frac{x}{ a - \lambda x } \exp \Big(  - \int_0^x \frac{f(y) }{a - \lambda y } dy  
\Big) dx =: p_a \Gamma_2(a).
$$
Hence $a$ solves  $a=\lambda m_a +p_a$ if and only if $a/p_a - \lambda m_a/p_a =1$, i.e. 
$a\Gamma_1(a)-\lambda\Gamma_2(a)=1$, i.e.
$$ 
\Gamma (a) := a \Gamma_1 (a) - \lambda \Gamma_2 (a )  
=\int_0^{a/\lambda } \exp \Big(- \int_0^x \frac{f(y) }{a - \lambda y } dy  \Big) dx=1.
$$
But $\Gamma$ is continuous and strictly increasing, $\Gamma(0)=0$ and $\Gamma(\infty)=\infty$, 
so that the equation
$\Gamma(a)=1$ has exactly one solution $a_*$. Finally, we obviously have $\Gamma(\lambda)<1$, so that
$a_*>\lambda$.
\end{proof}

We are now able to give the

\begin{proof}[Proof of Theorem \ref{theo:41}]
Consider an invariant probability measure $g$, supported by $\R_+$, for the nonlinear SDE \eqref{eq:dynlimit}.
Let $Y_0 \sim g$ and consider $(Y_t)_{t\geq 0}$ solution to \eqref{eq:dynlimit}. Then for all $t\geq 0$,
$Y_t \sim g$, so that $\E[Y_t]=m$ and $\E[f(Y_t)]=p$, where $m=\int_0^\infty x g(dx)$ and 
$p=\int_0^\infty f(x) g(dx)$. Consequently, $(Y_t)_{t\geq 0}$ solves \eqref{La} with $a=p+\lambda m$.
Since $(Y_t)_{t\geq 0}$ is stationary, we deduce from Proposition \ref{SDElin} that $g=g_a$.
But of course we have the constraint that $a=p_a+\lambda m_a$, whence $a=0$ or $a=a_*$ by Lemma 
\ref{unimi}. Hence either
$g=\delta_0$ or $g=g_{a_*}$.

\vip

Consider now $Y_0\sim g$, with $g=\delta_0$ or $g=g_{a_*}$. Then the solution $(Y_t)_{t\geq 0}$ to \eqref{La}
(with $a=0$ or $a=a_*$) is stationary by Proposition \ref{SDElin}. Since furthermore 
$\E[\lambda Y_t + f(Y_t)]=\int_0^\infty [\lambda x + f(x)]g(dx)=\lambda m_a+p_a = a$ by Lemma \ref{unimi}
since $a=0$ or $a=a_*$, we conclude that $(Y_t)_{t\geq 0}$ also
solves  \eqref{eq:dynlimit}. Consequently, $g$ is an invariant measure for \eqref{eq:dynlimit}.

\vip

We thus have checked that \eqref{eq:dynlimit} has exactly two invariant probability distributions, which are
$\delta_0$ and $g_{a_*}$. Finally $g_{a_*}$ is indeed the probability measure $g$ defined in the statement
(where $p=p_{a_*}$ and $m=m_{a_*}$)
and we have that $m+p/\lambda=a_*/\lambda>1$.
\end{proof}

\section{Shape of the time-marginals and large-time behavior}\label{sec:7}

The aim of this section is to prove Theorem \ref{theo:limitdensity}
and Propositions \ref{paszero} and \ref{formal}.
We thus consider $\lambda\geq 0$, grant Assumptions \ref{ass:1} and \ref{ass:2} and suppose that 
$\E [ f^2(Y_0)]  < \infty$ and $\P(Y_0=0)<1$.  
We consider the unique solution $(Y_t)_{t\geq 0}$ to \eqref{eq:dynlimit}, we set 
$p_t=\E [ f(Y_t) ]$, $m_t=\E[Y_t]$ , $a_t= \lambda m_t+ p_t$ and denote by $g(t)$ the law of $Y_t$.
We also recall that for $x \in [0,\infty)$ and $0\leq s<t$, $\varphi_{s, t } (x) = e^{ - \lambda (t-s) } x 
+ \int_s^t e^{ - \lambda ( t-u) } a_u du$. We also introduce 
\begin{equation}\label{eq:kappa}
\kappa_{s,t}(x)=\exp \Big(-\int_s^t f(\varphi_{s,u}(x))du\Big).
\end{equation}
Notice that $\varphi$ satisfies the flow property: 
one can directly check that for all $0 \leq r \leq s \leq t$, all $x \in [0,\infty)$,
$\varphi_{r,t}(y) = \varphi_{s,t}(\varphi_{r,s}(y))$.

\subsection{Time-marginals}

Let us first proceed to a few technical considerations.

\begin{lem}\label{cont}
Under the above conditions, 

\vip

(i) $t\mapsto m_t$, $t\mapsto p_t$ and $t\mapsto a_t$ are locally Lipschitz continuous
on $[0,\infty)$,

\vip

(ii) for all $t>0$, $\lim_{h\downarrow 0} h^{-1}\E[1-\kappa_{t-h,t}(Y_{t-h})]=p_t$.
\end{lem}

\begin{proof}
Using Remark \ref{rk1}-(ii) and \eqref{ttl1}, we observe that $t\mapsto \E[f^2(Y_t)]$ is locally bounded.
We now prove (i).
By the It\^o formula, we have $m_t=m_0+\intot \E[(1-Y_s)f(Y_s)]ds$ and 
$p_t=p_0+\intot \E[f'(Y_s)(p_s+\lambda (m_s-Y_s)) - f^2(Y_s)] ds$. 
But under Assumptions \ref{ass:1} and \ref{ass:2}, there is $C>0$ 
such that $x+f'(x) \leq C(1+f(x))$. We easily conclude that 
$s\mapsto \E[(1-Y_s)f(Y_s)]$ and $s\mapsto \E[f'(Y_s)(p_s+\lambda (m_s-Y_s)) - f^2(Y_s)]$ are locally bounded.
The conclusion follows.

\vip

We next fix $t>0$ and prove (ii). We write $|p_t-h^{-1}\E[1-\kappa_{t-h,t}(Y_{t-h})]| \leq \Delta^1_h + \Delta^2_h 
+\Delta^3_h$, where 
\begin{align*}
\Delta^1_h :=& |p_t - p_{t-h}|,\\
\Delta^2_h :=&  h^{-1} \Big|\E\Big[ h f(Y_{t-h}) - \int_{t-h}^t f(\varphi_{t-h,u}(Y_{t-h}))du \Big]\Big|,\\
\Delta^3_h :=&  h^{-1} \Big|\E\Big[\int_{t-h}^t f(\varphi_{t-h,u}(Y_{t-h}))du - (1-\kappa_{t-h,t}(Y_{t-h})) \Big]\Big|.
\end{align*}
First, we have $\lim_{h\downarrow 0} \Delta^1_h =0$ by point (i). Next, we see that for $h\in (0,t\land 1]$, for
$u\in[t-h,t]$ and for $x\geq 0$, it holds that
$\varphi_{t-h,u}(x) \leq x+C_th$ and 
$|x-\varphi_{t-h,u}(x)|\leq C_t (1+x) h$, for some constant $C_t$.
Hence
$|f(x)-f(\varphi_{t-h,u}(x))| \leq (\sup_{[0,x+C_t h)} f' ) \times C_t(1+x)h$.  Using Assumption \ref{ass:2}
and Remark \ref{rk1}-(ii), we get that $|f(x)-f(\varphi_{t-h,u}(x))| \leq C_t (1+f(x))(1+x)h$.
All this implies that $\Delta^2_h  \leq C_t h \E[(1+Y_{t-h})(1+f(Y_{t-h}))] \leq C_t h$
(because, as already seen, $s\mapsto \E[Y_sf(Y_s)]$ is locally bounded), 
which tends to $0$ as $h\downarrow 0$.
Finally, since $|y-(1-\exp(-y))| \leq y^2$ for all $y\geq 0$,
$$
\Delta^3_h \leq h^{-1}\E\Big[ \Big(\int_{t-h}^t f(\varphi_{t-h,u}(Y_{t-h}))du \Big)^2 \Big]
\leq \E\Big[ \int_{t-h}^t f^2(\varphi_{t-h,u}(Y_{t-h}))du \Big].
$$
As previously, we use Remark \ref{rk1}-(ii) to get $f^2(\varphi_{t-h,u}(x))\leq f^2(x+C_th) \leq C_t(1+f^2(x))$
(if $h \in (0,t\land 1]$), whence $\Delta^3_h \leq C_t h \E[1+f^2(Y_{t-h})]\leq C_t h$
(since  $s\mapsto \E[f^2(Y_s)]$ is locally bounded), which tends to $0$ as  $h\downarrow 0$.
This completes the proof.
\end{proof}

We now introduce, for $t\geq 0$, $\tau_t := \sup \{ s\in [0,t] \, : \, \Delta Y_s \ne 0 \}$, 
the last jump instant before $t$. We adopt the convention that $\sup \emptyset = 0$: 
if there is no jump during $[0,t]$, we set $\tau_t=0$. 

\begin{lem}\label{exp} 
Under the above conditions, 

\vip

(i) a.s., for all $t\geq 0$, $Y_t = \varphi_{0,t}(Y_0)\indiq_{ \{ \tau_t=0 \} } 
+ \varphi_{\tau_t,t}(0)\indiq_{ \{ \tau_t>0 \} }$,

\vip

(ii) for all $t>0$, $\P(\tau_t=0 \; \vert\; Y_0)=\kappa_{0,t}(Y_0)$,

\vip

(iii) for all $t>0$, $\P(Y_t=0)<1$.
\end{lem}

\begin{proof}
From \eqref{eq:dynlimit}, we have 
$Y_r = Y_0\indiq_{\{\tau_t=0\}}+ \int_{\tau_t}^r (a_s-\lambda Y_s)ds$ for 
all $t\geq 0$ and all $r\in [\tau_t,t]$. Solving this ODE, we find
$Y_t = e^{ - \lambda t } Y_0 \indiq_{\{ \tau_t = 0 \}} + \int_{\tau_t}^t e^{ - \lambda (t - s ) } a_s ds$, which proves point (i).
But $\tau_t=0$ implies that $\tau_s=0$ for all $s\in [0,t]$, whence $Y_s=\varphi_{0,s}(Y_0)$ on $[0,t]$.
As a consequence, $\{\tau_t=0\}=\{\intot\int_0^\infty \indiq_{\{z\leq f( \varphi_{0,s}(Y_0))\}}
\bN(ds,dz)=0\}$,
so that $\P(\tau_t=0 \; \vert\; Y_0) = \exp(-\intot f( \varphi_{0,s}(Y_0)) ds)$,
as claimed in point (ii). Using that  $Y_t \geq Y_0 e^{-\lambda t}$
on $\{\tau_t=0\}$ and point (ii), we see that
$$
\P(Y_t>0) \geq \P(Y_0>0,\tau_t=0) = \E[\kappa_{0,t}(Y_0)\indiq_{\{ Y_0 >0\}}]>0
$$
since $Y_0>0$ occurs with positive probability. This proves (iii).
\end{proof}

The law of $\tau_t $ is absolutely 
continuous on $ (0, t], $ as shown in the next proposition. This smoothness property will allow us
to show that jumps indeed create a density for $Y_t$.

\begin{prop}\label{loidetau}
Under the above conditions, for all $t>0$, the law  of $\tau_t$ is given by
$h_t(ds)=\E[\kappa_{0,t}(Y_0)] \delta_0 (ds) + p_s\kappa_{s,t}(0)\indiq_{\{0<s<t\}} ds$.
\end{prop}

\begin{proof}
First, $\P(\tau_t=0)=\E[\kappa_{0,t}(Y_0)]$ as desired by Lemma \ref{exp}-(ii). 
We next introduce the filtration $\cF_s=\sigma(\{Y_0,\bN([0,r]\times A) \, : \, r \in [0,s], 
A \in \cB([0,\infty))\} )$ and the process
$J_s=\sum_{r\in[0,s]} \indiq_{\{\Delta Y_r \ne 0\}} 
= \int_0^s \int_0^\infty \indiq_{\{z\leq f( Y_{r-})\}} \bN(dr,dz)$ which counts the number of jumps of $Y$.
We consider $0<s-h<s<t$ and observe that $\{\tau_t \in (s-h,s]\}=\{J_s>J_{s-h}\} \cap \{J_t=J_s\}$.
The event $\{J_s>J_{s-h}\}$ is $\cF_s$-measurable. When $Y$ does not jump during $(s,t]$,
$Y_{r-}=Y_r=\varphi_{s,r} (Y_s)$ for all $r\in (s,t]$: this follows from 
Lemma \ref{exp}-(i), from the semi-group property of the flow
$\varphi$, and from the fact that $\tau_r=\tau_s=\tau_t$ when  $Y$ does not jump during $(s,t]$.
Consequently,
$\{J_t=J_s\}=\{\int_{s}^t \int_0^\infty \indiq_{\{z\leq f( \varphi_{s,r} (Y_s))\}} \bN(dr,dz)=0\}$, whence
$$
\P( J_t=J_s \; \vert \; \cF_s) = \P \Big(\int_{s}^t \int_0^\infty \indiq_{\{z\leq f( \varphi_{s,r} (Y_s))\}} \bN(dr,dz)=0
\; \Big\vert \; \cF_s \Big)= \kappa_{s,t}(Y_s).
$$
We conclude that $\P( \tau_t \in (s-h,s]) = \E [\kappa_{s,t}(Y_s) \indiq_{\{J_s>J_{s-h}\}}]$.
On the event $\{J_s>J_{s-h}\}$, the process $Y$ jumps (at least once) to $0$ during $(s-h,s]$, so that
$Y_s \in [0, \varphi_{s-h,s}(0)]$ by Lemma \ref{exp}-(i). Hence, 
$$
|\P( \tau_t \in (s-h,s]) - \E [\kappa_{s,t}(0) \indiq_{\{J_s>J_{s-h}\}}]| 
\leq \sup_{x \in[0,\varphi_{s-h,s}(0)]}|\kappa_{s,t}(x)-\kappa_{s,t}(0)| \times \E[J_s-J_{s-h}].
$$
Using that $\E[J_s-J_{s-h}]=\int_{s-h}^s p_rdr \leq C h$ (by Lemma \ref{cont}-(i)), that  $\varphi_{s-h,s}(0)\to 0$
as $h\to 0$, and the (obvious) continuity of $x\mapsto \kappa_{s,t}(x)$, we conclude that
\begin{equation}\label{jab1}
\limsup_{h\to 0} \frac 1 h\Big|\P( \tau_t \in (s-h,s]) - \kappa_{s,t}(0) \P(J_s>J_{s-h})\Big| =0.
\end{equation}
Next, arguing exactly as in Lemma \ref{exp}-(ii), 
we get $\P( J_s>J_{s-h})=1-\E[\kappa_{s-h,s}(Y_{s-h})]$. Hence, we deduce from Lemma \ref{cont}-(ii) that
\begin{equation}\label{jab2}
\lim_{h\to 0} \frac 1 h \P( J_s>J_{s-h} )  = \E[f(Y_s)]=p_s.
\end{equation}
Gathering \eqref{jab1} and \eqref{jab2}, we deduce that indeed, the density of the law of
$\tau_t$ at point $s\in (0,t)$ exists and equals  $p_s \kappa_{s,t}(0)$.
\end{proof}

We are now able to give the

\vip

\begin{proof}[Proof of Theorem \ref{theo:limitdensity}.]
We have already seen that $t\mapsto a_t$ and $t\mapsto p_t$ are continuous (by Lemma \ref{cont}-(i))
and positive (by Lemma \ref{exp}-(iii)).
We now fix $t>0$. By Lemma \ref{exp}-(i),
$Y_t = \varphi_{0,t}(Y_0)\indiq_{ \{ \tau_t=0 \} } + \varphi_{\tau_t,t}(0)\indiq_{ \{ \tau_t>0 \}}$.
Hence for any bounded measurable $\phi: [0,\infty)\mapsto \R$,
\begin{align*}
\E [ \phi (Y_t) ] =& \E[ \phi(\varphi_{\tau_t,t}(0))\indiq_{ \{ \tau_t>0 \} }] + 
\E[\phi(\varphi_{0,t}(Y_0))\indiq_{ \{ \tau_t=0 \} }]=:A(\phi)+B(\phi).
\end{align*}
Clearly, $\varphi_{\tau_t,t}(0)<\varphi_{0,t}(0)$ when $\tau_t>0$ and $\varphi_{0,t}(Y_0)\geq \varphi_{0,t}(0)$.
Using Proposition \ref{loidetau}, we can write
$$
A(\phi)= \int_0^t   \phi ( \varphi_{ {s , t}} (0) ) p_s\kappa_{s,t}(0) ds.
$$
Recall that for $y \in [0,\varphi_{0,t}(0)]$, $\beta_t(y)\in[0,t]$ is uniquely defined
by $\varphi_{\beta_t(y),t}(0)=y$. The change of variables $s\mapsto y=\varphi_{ {s , t}} (0)$, for which
$dy = - e^{ - \lambda (t-s) } a_{s} ds$ and $s=\beta_t(y)$, gives us
$$
A(\phi)  =  \int_0^{ \varphi_{0, t } ( 0) } \phi ( y ) \frac{p_{\beta_t(y)}}{a_{\beta_t(y)}}\kappa_{\beta_t(y),t}(0) 
e^{\lambda(t-\beta_t(y))} dy.
$$
Next recall that $\gamma_t(y)=(y-\varphi_{0,t}(0))e^{\lambda t}$ for 
$y \geq \varphi_{0,t}(0)$. Using Lemma \ref{exp}-(ii) and the change of variables $x\mapsto y=\varphi_{0,t}(x)$,
for which $\gamma_t(y)=x$, we find
$$
B(\phi) = \E[\phi(\varphi_{0,t}(Y_0))\kappa_{0,t}(Y_0)]= \int_0^\infty \phi(\varphi_{0,t}(x))\kappa_{0,t}(x)g_0(dx)
=\int_{\varphi_{0,t}(0)}^\infty \phi(y) \kappa_{0,t}(\gamma_t(y)) (g_0 \circ \gamma_t^{-1})(dy).
$$
We have proved that 
$$
\E [ \phi (Y_t) ]= \int_0^\infty  \phi ( y ) \Big[\frac{p_{\beta_t(y)}}{a_{\beta_t(y)}}\kappa_{\beta_t(y),t}(0) 
e^{\lambda(t-\beta_t(y))}\indiq_{\{y < \varphi_{0,t}(0)\}} dy + 
\kappa_{0,t}(\gamma_t(y)) \indiq_{\{y\geq \varphi_{0,t}(0)\}} (g_0 \circ \gamma_t^{-1})(dy)   \Big].
$$
Replacing $\kappa$ by its expression, one finds the formula claimed in the statement.
\end{proof}

\subsection{Non extinction}

We first consider the easy case where $\lambda=0$.

\begin{proof}[Proof of Proposition \ref{paszero} when $\lambda=0$.]
Using Theorem \ref{theo:limitdensity}, we see that the law
of $Y_t$ has a density bounded by $1$, on $[0,\varphi_{0,t}(0))$ (because
for all $x\in[0,\varphi_{0,t}(0))$, there is $s\in(0,t]$ such that $\varphi_{s,t}(0)=x$).
For all $t> 1$, $\varphi_{0,t}(0)=\intot a_s ds > \int_0^1 a_s ds =:\alpha>0$ 
by Lemma \ref{exp}-(iii).
Hence, for $\e:= \min\{\alpha,1/2\}$, it holds that $\Pr(Y_t \leq \e) \leq \e\leq 1/2$ for all $t>1$.
Consequently, $Y_t$ cannot tend to $0$ in probability as $t\to \infty$.
\end{proof}

To study the case where $\lambda>0$, we need the following lemma.

\begin{lem}\label{estap2}
Grant Assumptions \ref{ass:1}, \ref{ass:2}, \ref{ass:4}-(i) and suppose that $\lambda \geq 0$.
Assume that $\E[f(Y_0)]<\infty$ and consider the unique solution $(Y_t)_{t\geq 0}$ to
\eqref{eq:dynlimit} built in Theorem \ref{theo:2}-(iii).
There is $q>1$ such that
$\sup_{t\geq t_0} \E[f^q(Y_t)] <\infty$ for all $t_0>0$.
\end{lem}

\begin{proof} 
By Assumption \ref{ass:4}-(i), there exists $q>1$ such that $q \limsup_{x \to \infty } [f' (x) / f(x)] < 1$. 
We now divide the proof into two steps.

\vip

{\it Step 1.} Here we assume that $Y_0$ has a bounded support, so that $Y_t$ 
is uniformly bounded on each compact time interval by \eqref{ttl1}.
This ensures that all the computations below are licit.
Applying the It\^o formula for jump processes and taking expectations, we easily check that
\begin{align*}
\frac{d}{dt}\E[f^q(Y_t)] =& -\lambda \E[(f^q)'(Y_t)(Y_t-\E[Y_t])] - \E[f^{q+1}(Y_t)] + 
\E[(f^q)'(Y_t)]\E[f(Y_t)].
\end{align*}
Now, as shown e.g. in \cite{sc}, it holds that $\E[\varphi_1(X)]\E[\varphi_2(X)]\leq \E[\varphi_1(X)\varphi_2(X)]$
for any $[0,\infty)$-valued random variable $X$ and any pair of non-decreasing functions $\varphi_1,\varphi_2:
[0,\infty)\mapsto \R$. Using that $f^q$ is convex (since $q>1$ and by Assumption \ref{ass:2}), we deduce
that $\E[(f^q)'(Y_t)(Y_t-\E[Y_t])]\geq 0$ and that $\E[(f^q)'(Y_t)]\E[f(Y_t)] \leq \E[(f^q)'(Y_t) f(Y_t)]$.
We thus have
\begin{align*}
\frac{d}{dt}\E[f^q(Y_t)] \leq & - \E[f^{q+1}(Y_t) - (f^q)'(Y_t) f(Y_t)].
\end{align*}
Recalling now that $q \limsup_{x \to \infty } [f' (x) / f(x)] < 1$, we can find two constants
$c>0$ and $C\geq 0$ such that $f^{q+1}(x) - (f^q)'(x) f(x) \geq c f^{q+1}(x)-C$ for all $x\geq 0$. 
Indeed, find $x_0>0$ such that
$a:=\sup_{x\geq x_0} [f' (x) / f(x)] < 1/q$, observe that $f^{q+1}(x) - (f^q)'(x) f(x)
=f^{q+1}(x) - q f'(x) f^q(x) \geq (1-aq)f^{q+1}(x)$ for $x\geq x_0$, and conclude by setting $c=(1-aq)>0$
and $C=\sup_{[0,x_0]} (f^q)'(x) f(x)$. We deduce that 
\begin{align*}
\frac{d}{dt}\E[f^q(Y_t)] \leq & C-c \E[f^{q+1}(Y_t)] \leq C -c\E[f^{q}(Y_t)]^{1+1/q}.
\end{align*}
The conclusion classically follows: there is a constant $K$, not depending on $\E[f^q(Y_0)]$
such that $\E[f^q(Y_t)] \leq K(1+t^{-q})$.

\vip

{\it Step 2.} We next only assume that $\E[f(Y_0)]<\infty$. We introduce $Y_0^A=\min\{Y_0,A\}$
and the unique solution $(Y_t^A)_{t\geq 0}$ to \eqref{eq:dynlimit} starting from $Y_0^A$.
By Step 1, we know that for all $t\geq 0$, uniformly in $A$,  $\E[f^q(Y_t^A)] \leq K(1+t^{-q})$.
But we also know by Proposition \ref{unipascomp} that for each $t\geq 0$,
$Y_t^A$ goes in law to $Y_t$ as $A\to \infty$ (we have to verify that
$\E[|H(Y_0^A)-H(Y_0)|]\to 0$, which is not difficult by dominated convergence since $\E[H(Y_0)]
\leq \pi/2 + \E[f(Y_0)]<\infty$ by assumption). The conclusion follows.
\end{proof}

\begin{proof}[Proof of Proposition \ref{paszero} when $\lambda>0$.]
We work by contradiction and assume that $Y_t$ goes in law (and thus in probability) to $0$ as $t\to\infty$.
By Lemma \ref{estap2}, we know that there is $q >1$ such that $\sup_{t\geq 1} \E[f^{q}(Y_t)] <\infty$.
This implies that $\sup_{t\geq 1} \E[Y_t^{q}] <\infty$ by Remark \ref{rk1}-(i).
Consequently,
$Y_t$ and $f(Y_t)$ are uniformly integrable (for $t\geq 1$), so that the Lebesgue theorem tells us,
since $Y_t$ goes in probability to $0$, that $a_t=\lambda\E[Y_t] + E[f(Y_t)]$ tends to $0$.

\vip

We use Lemma \ref{exp}-(i) to write $Y_t= e^{ - \lambda t } Y_0 \indiq_{\{ \tau_t = 0 \}} + \varphi_{\tau_t,t}(0)$.
First, there is $t_0>0$ such that 
\begin{equation}\label{majo}
\hbox{for all $t\geq t_0$,} \quad \varphi_{\tau_t,t}(0) \leq \varphi_{0,t}(0) \leq 1/2.
\end{equation}
Indeed, we consider $t_1>0$ such that 
$a_t \le \lambda/4$ for all $t\geq t_1$. Then we see that $\varphi_{\tau_t,t}(0) \leq \varphi_{0,t}(0)
\leq e^{-\lambda t}\int_0^{t_1} e^{\lambda u}a_u du + (\lambda/4)\int_{t_1}^t e^{-\lambda (t-u)} du \leq
C e^{-\lambda t} + 1/4$, whence the conclusion.

\vip

Taking expectations in \eqref{eq:dynlimit}, we see that
$(d/dt) \E [Y_t]=\E[(1-Y_t)f(Y_t)]$. But for $t\geq t_0$ and on the event 
$\{\tau_t>0\}$, we have $Y_t=\varphi_{\tau_t,t}(0) \leq 1/2$ and thus $(1-Y_t)f(Y_t)\geq 0$.
Next on the event $\{\tau_t=0\}$, we have $Y_t=\varphi_{0,t}(Y_0)$.
Consequently,
$$
\frac{d}{dt} \E [Y_t] \geq \E\big[(1-\varphi_{0,t}(Y_0))
f(\varphi_{0,t}(Y_0))\indiq_{\{\tau_t=0\}}\big].
$$
Hence, by Lemma \ref{exp}-(ii), 
\begin{align*}
\frac{d}{dt} \E [Y_t] \geq \E[(1-\varphi_{0,t}(Y_0))
f(\varphi_{0,t}(Y_0)) \kappa_{0,t}(Y_0)] \geq  I_t-J_t,
\end{align*}
where
\begin{align*}
I_t:= & \E[(1-\varphi_{0,t}(Y_0)) f(\varphi_{0,t}(Y_0)) \kappa_{0,t}(Y_0) \indiq_{\{\varphi_{0,t}(Y_0)<3/4\}}],\\
J_t:= & \E[(\varphi_{0,t}(Y_0)-1) f(\varphi_{0,t}(Y_0)) \kappa_{0,t}(Y_0) \indiq_{\{\varphi_{0,t}(Y_0)>1\}}].
\end{align*}
We will now prove that there is $t_2>t_0$ such that  $I_t > J_t$ for all $t\geq t_2$.
This will end the proof, since then $(d/dt)\E[Y_t]> 0$ for all $t\geq t_2$, so that 
$\E[Y_t]$, and thus {\it a fortiori} $a_t$, cannot go to $0$. 
To prove that $I_t$ is eventually greater than $J_t$, we will check that, with 
$\xi\geq 1$ defined in Assumption \ref{ass:4}-(ii),
\begin{align*}
\hbox{(a)}\;\; \liminf_{t\to \infty}  \frac{e^{\lambda \xi t}}{\kappa_{0,t}(0)}I_t >0,\quad 
\hbox{(b)} \;\; \limsup_{t\to \infty}  \frac{e^{\lambda \xi t}}{\kappa_{0,t}(0)}J_t =0 .
\end{align*}

Let us first prove (b). It holds that $\kappa_{0,t}(Y_0)/\kappa_{0,t}(0) \leq 1$.
Next, recalling that $\varphi_{0,t}(Y_0)=e^{-\lambda t}Y_0 + \varphi_{0,t}(0)\leq e^{-\lambda t}Y_0+1/2$ for $t\geq t_0$ by 
\eqref{majo}, we see that $\indiq_{\{\varphi_{0,t}(Y_0)>1\}}\leq \indiq_{\{e^{-\lambda t}Y_0 >1/2\}}$.
On the set $e^{-\lambda t}Y_0 >1/2$, it holds that $\varphi_{0,t}(Y_0) \leq e^{-\lambda t}Y_0+1/2 \leq 2 e^{-\lambda t}Y_0$.
Finally, by Assumption \ref{ass:4}-(ii), we write that $e^{\lambda \xi t} (\varphi_{0,t}(Y_0)-1) f(\varphi_{0,t}(Y_0))
\leq C e^{\lambda \xi t}[ (2e^{-\lambda t}Y_0)^\xi + (2e^{-\lambda t}Y_0)^{\zeta+1}] \leq C [Y_0^\xi+Y_0^{\zeta+1}]
\leq C(1+Y_0^{\zeta+1})$ because $\zeta \geq \xi-1$ by assumption. All in all, we have checked that for $t\geq t_0$,
$$
\frac{e^{\lambda \xi t}}{\kappa_{0,t}(0)}J_t \leq C \E[(1+Y_0^{\zeta+1})\indiq_{\{Y_0 >e^{\lambda t}/2\}}].
$$
Since $\E[Y_0^{\zeta+1}]<\infty$ by assumption, the monotone convergence theorem shows the validity of (b).

\vip

We finally prove (a). We have $\varphi_{0,t}(Y_0)\geq e^{-\lambda t}Y_0$.
Hence by Assumption \ref{ass:4}-(ii), we may write, on the event $\{\varphi_{0,t}(Y_0)<3/4\}$,
that $(1-\varphi_{0,t}(Y_0)) f(\varphi_{0,t}(Y_0)) \geq c e^{-\lambda \xi t}Y_0^\xi$.
We next recall that, as previously, for $t\geq t_0$,
$\varphi_{0,t}(Y_0)\leq e^{-\lambda t}Y_0+1/2$. Consequently,
$\indiq_{\{\varphi_{0,t}(Y_0)<3/4\}}\geq \indiq_{\{e^{-\lambda t}Y_0 <1/4\}}$.
Finally, 
\begin{align*}
\frac{\kappa_{0,t}(Y_0)}{\kappa_{0,t}(0)}=& \exp\Big(-\intot  [f(\varphi_{0,s}(Y_0)) -f(\varphi_{0,s}(0))] ds   \Big)\\
\geq &  \exp\Big(-\intot e^{-\lambda s} Y_0 \times\big( \sup_{[0,\varphi_{0,s}(Y_0)]} f' \big) ds \Big),
\end{align*}
because $\varphi_{0,s}(Y_0)=e^{-\lambda s} Y_0+\varphi_{0,s}(0)$.
But $\sup_{s \geq 0}\varphi_{0,s}(Y_0) \leq Y_0 +C$: this follows from the expression of $\varphi$
and the fact that $t\mapsto a_t$ is a 
bounded function, since it is locally bounded and tends to $0$. Hence, using Assumptions \ref{ass:1}, 
\ref{ass:2} and Remark \ref{rk1}-(ii), 
$\sup_{[0,\varphi_{0,s}(Y_0)]} f' \leq C (1+f(Y_0))$, and we end with
\begin{align*}
\frac{\kappa_{0,t}(Y_0)}{\kappa_{0,t}(0)}
\geq &  \exp\big(- C Y_0(1+f(Y_0)) \big).
\end{align*}
All in all, we see that for $t\geq t_0$,
$$
\frac{e^{\lambda \xi t}}{\kappa_{0,t}(0)}I_t \geq c 
\E\big[Y_0^\xi\exp\big(- C Y_0(1+f(Y_0)) \big)\indiq_{\{Y_0 < e^{\lambda t}/4\}}\big].
$$
This last quantity tends, by monotone convergence, to 
$c \E[Y_0^\xi\exp\big(- C Y_0(1+f(Y_0)) \big)]>0$.  We have checked (a).
\end{proof}

\subsection{Trend to equilibrium when $\lambda=0$}\label{xxx}
This final part is dedicated to the proof of Proposition \ref{formal}.
We thus work under all the assumptions above and suppose furthermore that
$\lambda=0$ and that the initial condition $g_0$ has a density $g_0\in C^1_b([0,\infty))$ satisfying
$g_0(0)=1$ and $\int_0^\infty |g_0'(y)|dy <\infty.$

\vip

Since $\lambda=0$, we simply have $\varphi_{s,t}(x)=x+A_t-A_s$, where $A_t=\intot a_s ds$.
For $t\geq 0$ and $y\in [0,A_t]$, $\beta_t(y) \in [0,t]$ is defined by $A_t-A_{\beta_t(y)}=y$.
And for $t\geq 0$ and $y\geq A_t$, $\gamma_t(y)=y-A_t$.
We thus know from Theorem \ref{theo:limitdensity} that $g(t)$ has a density on $[0,\infty)$
given by
\begin{equation}\label{tacc}
g(t,y)=\kappa_{\beta_t(y),t}(0) \indiq_{\{y < A_t\}} +  \kappa_{0,t}(y-A_t) g_0(y-A_t)\indiq_{\{y\geq A_t\}}.
\end{equation}
Observe that $g(t,A_t-)=g(t,A_t+)=\kappa_{0,t}(0)$. A little study, using 
our assumptions on $g_0$ and that the map $t\mapsto a_t$ is continuous and positive, shows that
$g(t,y)$ is continuous on $[0,\infty)^2$, of class $C^1$ on
$\{(t,y)\in [0,\infty)^2\; : \;  y \ne A_t\}$ and that
$\sup_{[0,T]} (||\partial_y g(t,.)||_{L^\infty(\R)}+ ||\partial_y g(t,.)||_{L^1(\R)}) <\infty$ for all $T>0$.

\vip

Let $\phi \in C^1_b([0,\infty))$. Applying the It\^o formula to compute $\phi(Y_t)$, taking expectations
and differentiating the obtained formula, we find that 
$$
\frac{d}{dt} \int_0^\infty \phi (x) g(t, x) dx = \int_0^\infty [\phi(0)-\phi(x) +p_t \phi'(x)]g(t,x)dx=
- \int_0^\infty \phi( x) \left[  f(x)  g(t, x) + p_t \partial_x g(t, x)   \right]  dx.
$$
The second equality follows from an integration by parts, which is licit
because for $t$ fixed, $g(t,.)$ is continuous, piece-wise $C^1$ and $\int_0^\infty |\partial_y g(t,y)| dy <\infty$.
The boundary term disappears since $g(t,0)=1$.

\vip

We introduce now $g(x)=\exp(-p^{-1}\int_0^x f(y)dy)$ as in Theorem \ref{theo:41}-(ii), which
solves $p \partial_x g(x) + f(x)g(x)=0$.
Hence, for any $\phi \in C^1_b,$ 
\begin{align*}
 \frac{d}{dt}\int_0^\infty \phi( x) (g(t,x)-g(x))dx 
=& -\int_0^\infty \phi(x)  f(x)(g(t,x)-g(x)) dx 
- p_t\int_0^\infty \phi(x)  (\partial_x (g(t,x)-g(x))dx\\
&+ (p-p_t)\int_0^\infty \phi(x) \partial_x g(x) dx .
\end{align*}
We thus can apply Lemma \ref{dpl} below
with $a(t,x)=g(t,x)-g(x)$ and $b(t,x)= -f(x)(g(t,x)-g(x))- p_t \partial_x (g(t,x)-g(x))
+ (p-p_t)\partial_x g(x)$, which both belong to $L^\infty_{loc}([0,T],L^1([0,\infty))$
because $(1+f(x))g(x)+|\partial_x g(x)|$ is integrable, because $p_t$ is locally bounded,
because $\int_0^\infty f(x)g(t,x)dx=p_t$ and because $\sup_{[0,T]} ||\partial_x g(t,\cdot)||_{L^1(\R)} <\infty$
for all $T>0$, to deduce that $t\mapsto \int_0^\infty|g(t,x)-g(x)|dx$ is continuous and satisfies,
for a.e. $t\geq 0$,
\begin{align*}
\frac d{dt} \int_0^\infty|g(t,x)-g(x)|dx=& - \int_0^\infty  \sg(g(t,x)-g(x))f(x)(g(t,x)-g(x))dx \\
&- p_t \int_0^\infty  \sg(g(t,x)-g(x)) 
\partial_x (g(t,x)-g(x)) dx\\
&+(p-p_t) \int_0^\infty \sg(g(t,x)-g(x))\partial_x g(x).
\end{align*}
The second term on the RHS equals 
$-p_t \int_0^\infty \partial_x[|g(t,x)-g(x)|]dx = p_t |g(t,0)-g(0)|=0$. Using that $g$ is decreasing,
we easily check that the last term is bounded by 
$$
|p_t-p| \int_0^\infty |\partial_x g(x)| dx = |p-p_t| g(0)=|p-p_t| = \Big|\int_0^\infty f(x)(g(t,x)-g(x))dx \Big|.
$$
Thus, for a.e. $t\geq 0$,
\begin{align*}
\frac d{dt} ||g(t)-g||_{L^1} \leq & - \int_0^\infty f(x)|g(t,x)-g(x)|dx 
+ \Big|\int_0^\infty f(x)(g(t,x)-g(x))dx \Big| \\
=& - 2 \min \Big\{\int_0^\infty f(x)(g(t,x)-g(x))_+dx , \int_0^\infty f(x)(g(t,x)-g(x))_-dx    \Big\} .
\end{align*}
But by (\ref{tacc}),
$g(t,x) \leq 1$ for all $t\geq 1$, all $x\in [0,A_1]$ and $A_1=\int_0^1 a_sds>0$. Consequently, for $t\geq 1$,
all $\e \in [0,A_1]$, since $f$ is nondecreasing,
$$ 
\int_0^\infty f(x)(g(t,x)-g(x))_+dx \geq f(\e) \int_\e^\infty (g(t,x)-g(x))_+dx
\geq f(\e) \Big[\int_0^\infty (g(t,x)-g(x))_+dx - \e\Big].
$$
But $g(t)$ and $g$ being two probability density functions, $\int_0^\infty (g(t,x)-g(x))_+dx
=||g(t)-g||_{L^1}/2$. Thus for all $t\geq 1$, all $\e \in [0,A_1]$,
$$ 
\int_0^\infty f(x)(g(t,x)-g(x))_+dx \geq f(\e) \Big[||g(t)-g||_{L^1}/2 - \e\Big].
$$
Since $g(x) \leq 1$ for all $x\geq 0$, a similar estimate holds true for $\int_0^\infty f(x)(g(t,x)-g(x))_-dx$.
All in all, we have proved that for a.e. $t\geq 1$, all $\e \in [0,A_1]$,
\begin{align*}
\frac d{dt} ||g(t)-g||_{L^1} \leq & - f(\e) \Big[||g(t)-g||_{L^1} - 2\e \Big].
\end{align*}
Choosing $\e = \min\{A_1, ||g(t)-g||_{L^1}/4  \}$ and introducing the function
$\Phi(x)=x f(A_1 \land (x/4))/2$, this implies, still for a.e. $t\geq 1$, that
\begin{align*}
\frac d{dt} ||g(t)-g||_{L^1} \leq & - \Phi(||g(t)-g||_{L^1}).
\end{align*}
Since $\Phi$ is nonnegative increasing on $(0,\infty)$ and vanishes only at $0$, 
we easily conclude, using that $t\mapsto ||g(t)-g||_{L^1}$ is continuous, 
that indeed, $\lim_{t\to \infty} ||g(t)-g||_{L^1}=0$. Recalling that 
$ 2 \| g(t) - g \|_{TV} = ||g(t)-g||_{L^1} $ implies the result.

\vip

If now $f(x)\geq c x^\xi$ on $[0,1]$ for some $\xi \geq 1$, 
there clearly exists another constant $c>0$ such that $\Phi(x) \geq c x^{\xi+1}$
for all $x \in [0,2]$ (recall that $f$ is non-decreasing).
But $||g(t)-g||_{L^1}$ always belongs to $[0,2]$. We thus have, for a.e. $t\geq 1$,
$(d/dt) ||g(t)-g||_{L^1} \leq  - c ||g(t)-g||_{L^1}^{\xi+1}$. It is then not hard to deduce that
there is a constant $C$ such that $||g(t)-g||_{L^1} \leq  C (1+t)^{-1/\xi}$ for all $t\geq 0$.

\vip

It remains to check the following lemma, which is well-known folklore since the
seminal work of DiPerna and Lions \cite{dipernalions}. We unfortunately found no precise reference.

\begin{lem}\label{dpl}
Let $a,b: [0,\infty)\times \R \mapsto \R$ belong to $L^\infty_{loc}([0,\infty),L^1(\R))$, that is 
$\sup_{[0,T]} \int_\R (|a(t,x)|+|b(t,x)|) dx <\infty$ for all $T>0$. Assume that for all $\phi \in C^1_b(\R)$,
all $t\geq 0$,
\begin{equation}\label{ww}
\frac d{dt}\int_\R\phi(y)a(t,y)dy = \int_\R\phi(y)b(t,y)dy.
\end{equation}
Then $t\mapsto \int_\R |a(t,x)|dx$ is continuous and for a.e. $t\geq 0$,
$$
\frac{d}{dt}\int_\R |a(t,x)|dx = \int_\R \sg(a(t,x)) b(t,x) dx,
$$
where $\sg(u)=\indiq_{\{u>0\}}- \indiq_{\{u<0\}}$.
\end{lem}

\begin{proof}
We introduce $\rho_\e=(2\pi \e)^{-1/2} \exp(-x^2/(2\e))$ for $\e>0$ and $x \in \R$
and define $a_\e(t,x)=[a(t,\cdot)\star\rho_\e](x)=\int_\R a(t,y)\rho_\e(x-y)dy$ and 
$b_\e(t,x)=[b(t,\cdot)\star\rho_\e](x)$.
It is well-known that for all $t\geq 0$, 
$\lim_{\e \to 0}(||a(t,\cdot)-a_\e(t,\cdot)||_{L^1}+||b(t,\cdot)-b_\e(t,\cdot)||_{L^1})   =0$.
It is also clear that for all $t\geq 0$, all $\e>0$, $||a_\e(t,\cdot)||_{L^1}\leq ||a(t,\cdot)||_{L^1}$
and $||b_\e(t,\cdot)||_{L^1}\leq ||b(t,\cdot)||_{L^1}$.

\vip

{\it Step 1.}
Applying \eqref{ww} with $\phi(y)=\rho_\e(x-y)$, we find that for all $t\geq 0$, all $x \in \R$,
$\partial_ta_\e(t,x)=b_\e(t,x)$. Hence for any $\psi \in C^1_b(\R)$, $\partial_t \psi(a_\e(t,x))=\psi'(a_\e(t,x))
b_\e(t,x)$. We conclude that for all $t\geq 0$,
\begin{equation}\label{se}
\int_\R \psi(a_\e(t,x)) dx = \int_\R \psi(a_\e(0,x)) dx + \intot \int_\R\psi'(a_\e(s,x))b_\e(s,x) dxds.
\end{equation}

{\it Step 2.}
We now pass to the limit as $\e\to 0$ in \eqref{se} to deduce that for all $t\geq 0$, all $\psi \in C^2_b(\R)$,
\begin{equation}\label{ss}
\int_\R \psi(a(t,x)) dx = \int_\R \psi(a(0,x)) dx + \intot \int_\R\psi'(a(s,x))b(s,x) dxds.
\end{equation}
First, we clearly have that $\lim_{\e \to 0} \int_\R \psi(a_\e(t,x)) dx = \int_\R \psi(a(t,x)) dx$
because $\psi$ is globally Lipschitz continuous and because $\lim_{\e \to 0} ||a(t,\cdot)-a_\e(t,\cdot)||_{L^1}=0$.
The first term on the RHS is of course treated similarly.
We next introduce 
$\Delta_\e =|\intot \int_\R\psi'(a_\e(s,x))b_\e(s,x) dxds-\intot \int_\R\psi'(a(s,x))b(s,x) dxds|$
and write
\begin{align*}
\Delta_\e \leq& ||\psi'||_\infty \intot \int_\R |b_\e(s,x)-b(s,x)|dxds + \intot \int_\R |b(s,x)||\psi'(a_\e(s,x))
-\psi'(a(s,x))|dxds=I_\e+J_\e.
\end{align*}
First, $\lim_{\e\to 0} I_\e=0$ by dominated convergence, because $\int_\R |b_\e(s,x)-b(s,x)|dx$
is bounded on $[0,t]$ and tends to $0$ for each $s\geq 0$. Next, we observe that for all (large) $K>0$,
$$
J_\e \leq 2  ||\psi'||_\infty \intot \int_\R |b(s,x)| \indiq_{\{|b(s,x)|> K\}} dxds + ||\psi''||_\infty K
\intot \int_\R|a_\e(s,x)-a(s,x)|dxds.
$$
The second term tends to $0$ as $\e\to 0$, for the same reasons as for $I_\e$. We thus conclude that
$\limsup_{\e\to 0} \Delta_\e \leq 2  ||\psi'||_\infty \intot \int_\R |b(s,x)| \indiq_{\{|b(s,x)|> K\}} dxds$
for all $K>0$. But this last quantity tends to $0$ as $K\to \infty$, so that finally,
$\lim_{\e\to 0} \Delta_\e =0$ and \eqref{ss} is verified.

\vip

{\it Step 3.} Here we verify, and this will conclude the proof, that 
\begin{equation}\label{sss}
\int_\R |a(t,x)| dx = \int_\R |a(0,x)| dx + \intot \int_\R\sg(a(s,x)) b(s,x) dxds.
\end{equation}
We consider a sequence of even smooth nonnegative functions $\psi_n \in C^2_b(\R)$, such that
$\psi_n(u)$ increases to $|u|$, for each $u\in \R$, as $n\to \infty$, such that
$\psi_n'(u)$ tends to $\sg(u)$ for each $u\in\R$ and such that $\sup_n ||\psi_n'||_\infty \leq 2$.
The choice $\psi_n(u)=\sqrt{x^2+1/n}-\sqrt{1/n}$ is possible. By Step 2, we find,
for all $t\geq 0$, all $n\geq 1$,
$$
\int_\R \psi_n(a(t,x)) dx = \int_\R \psi_n(a(0,x)) dx + \intot \int_\R\psi_n'(a(s,x))b(s,x) dxds.
$$
By monotone convergence, we have $\lim_n \int_\R \psi_n(a(t,x)) dx=\int_\R |a(t,x)| dx$
and $\lim_n \int_\R \psi_n(a(0,x)) dx=\int_\R |a(0,x)| dx$. 
It also holds true that $\lim_n\intot \int_\R\psi_n'(a(s,x))b(s,x) dxds=\intot \int_\R\sg(a(s,x)) b(s,x) dxds$,
by dominated convergence. Indeed,  we know that $\psi_n'(a(s,x)) \to \sg(a(s,x))$ for each $s,x$, and
$|\psi_n'(a(s,x))b(s,x)| \leq 2 |b(s,x)|$, which is integrable on $[0,t]\times \R$ by assumption.
\end{proof}

\section*{Acknowledgments}

This paper is a continuation of \cite{aaee}, and EL thanks Marzio Cassandro, Anna De Masi and 
Errico Presutti for their collaborative
participation and the many discussions we had in L'Aquila and Rome, especially concerning the 
invariant distribution of the limit process. EL also thanks 
Carl Graham for many discussions at the beginning of the work. 
This work is part of  FAPESP project ``NeuroMat" (grant 2011/51350-6).
The authors warmly thank the referees for their suggestions that allowed them to improve subsequently
the presentation of the paper.

\end{document}